\documentclass[10pt]{amsart}
\usepackage{amsmath}
\usepackage{amscd,amsthm,amssymb,amsfonts}
\usepackage{mathrsfs}
\usepackage{dsfont}
\usepackage{stmaryrd}
\usepackage{euscript}
\usepackage{expdlist}
\usepackage{enumerate}

\input xy
\xyoption{all}
\usepackage[OT2,T1]{fontenc}

\theoremstyle{plain}
\newtheorem{thm}{Theorem}[section]
\newtheorem{thmA}{Theorem}

\newtheorem*{thm*}{Theorem}

\newtheorem{lm}[thm]{Lemma}

\newtheorem*{corb*}{Corollary B}
\newtheorem{prop}[thm]{Proposition}
\newtheorem*{conj*}{Conjecture}

\theoremstyle{remark}
\newtheorem*{remark}{Remark}
\newtheorem*{thank}{Acknowledgments}

\theoremstyle{definition}
\newtheorem*{defn*}{Definition}
\newtheorem{Remark}[thm]{Remark}
\newtheorem{I_Remark*}{Remark}
\newtheorem{defn}[thm]{Definition}

\newcommand{\nc}{\newcommand}

\newcommand{\beq}{\begin{equation}}
\newcommand{\eeq}{\end{equation}}
\newcommand{\bpmx}{\begin{pmatrix}}
\newcommand{\epmx}{\end{pmatrix}}
\newcommand{\bbmx}{\begin{bmatrix}}
\newcommand{\ebmx}{\end{bmatrix}}
\newcommand{\wh}{\widehat}
\newcommand{\wtd}{\widetilde}

\newcommand{\beqcd}[1]{\begin{equation*}\label{#1}\tag{#1}}
\newcommand{\eeqcd}{\end{equation*}}

\numberwithin{equation}{section}

\def\parref#1{\ref{#1}}
\def\thmref#1{Theorem~\parref{#1}}

\def\propref#1{Proposition~\parref{#1}}
     \def\remref#1{Remark~\parref{#1}}
\def\secref#1{\S\parref{#1}}

\def\lmref#1{Lemma~\parref{#1}}
\def\subsecref#1{\S\parref{#1}}

\def\makeop#1{\expandafter\def\csname#1\endcsname
  {\mathop{\rm #1}\nolimits}\ignorespaces}

\makeop{Hom}   \makeop{End}   \makeop{Aut}   
\makeop{Pic} \makeop{Gal}       \makeop{Div} \makeop{Lie}
\makeop{PGL}   \makeop{Corr} \makeop{PSL} \makeop{sgn} \makeop{Spf}
 \makeop{Tr} \makeop{Nr} \makeop{Fr} \makeop{disc}
\makeop{Proj} \makeop{supp} \makeop{ker}   \makeop{Im} \makeop{dom}
\makeop{coker} \makeop{Stab} \makeop{SO} \makeop{SL} \makeop{SL}
\makeop{Cl}    \makeop{cond} \makeop{Br} \makeop{inv} \makeop{rank}
\makeop{id}    \makeop{Fil} \makeop{Frac}  \makeop{GL} \makeop{SU}
\makeop{Trd}   \makeop{Sp} \makeop{Tr}    \makeop{Trd} \makeop{Res}
\makeop{ind} \makeop{depth} \makeop{Tr} \makeop{st} \makeop{Ad}
\makeop{Int} \makeop{tr}    \makeop{Sym} \makeop{can} \makeop{SO}
\makeop{torsion} \makeop{GSp} \makeop{Tor}\makeop{Ker} \makeop{rec}
\makeop{Ind} \makeop{Coker}
 \makeop{vol} \makeop{Ext} \makeop{gr} \makeop{ad}
 \makeop{Gr}\makeop{corank} \makeop{Ann}
\makeop{Hol} 
\makeop{Fitt} \makeop{Mp} \makeop{CAP}

\def\Ord{{\mathrm{ord}}}

\DeclareMathOperator{\GO}{GO}
\DeclareMathOperator{\GSO}{GSO}

\DeclareMathAlphabet{\mathpzc}{OT1}{pzc}{m}{it}
\DeclareSymbolFont{cyrletters}{OT2}{wncyr}{m}{n}
\DeclareMathSymbol{\SHA}{\mathalpha}{cyrletters}{"58}

\def\makebb#1{\expandafter\def
  \csname bb#1\endcsname{{\mathbb{#1}}}\ignorespaces}
\def\makebf#1{\expandafter\def\csname bf#1\endcsname{{\bf
      #1}}\ignorespaces}
\def\makegr#1{\expandafter\def
  \csname gr#1\endcsname{{\mathfrak{#1}}}\ignorespaces}
\def\makescr#1{\expandafter\def
  \csname scr#1\endcsname{{\EuScript{#1}}}\ignorespaces}
\def\makecal#1{\expandafter\def\csname cal#1\endcsname{{\mathcal
      #1}}\ignorespaces}

\def\doLetters#1{#1A #1B #1C #1D #1E #1F #1G #1H #1I #1J #1K #1L #1M
                 #1N #1O #1P #1Q #1R #1S #1T #1U #1V #1W #1X #1Y #1Z}
\def\doletters#1{#1a #1b #1c #1d #1e #1f #1g #1h #1i #1j #1k #1l #1m
                 #1n #1o #1p #1q #1r #1s #1t #1u #1v #1w #1x #1y #1z}
\doLetters\makebb   \doLetters\makecal  \doLetters\makebf
\doLetters\makescr
\doletters\makebf   \doLetters\makegr   \doletters\makegr

\normalsize

\makeop{Ram} \makeop{Rep} \makeop{mass}

\makeop{Bl}
\def\abs#1{\left|#1\right|}

\def\Qp{\Q_p}
\def\Qbar{\ol{\Q}}
\def\Zbar{\ol{\Z}}

\def\Zp{\Z_p}

\def\rmN{{\mathrm N}}

\def\cA{{\mathcal A}}  

\def\cE{{\mathcal E}}
\def\cF{{\mathcal F}}  

\def\cL{{\mathcal L}}
\def\cH{{\mathcal H}}

\def\cR{{\mathcal R}}
\def\cO{\mathcal O}
\def\cS{{\mathcal S}}
\def\cf{{\mathcal f}}
\def\cW{{\mathcal W}}

\def\cP{{\mathcal P}}

\def\cN{\mathcal N}

\def\cU{\mathcal U}

\def\bfK{\mathbf K}

\def\bff{\mathbf f}

\def\bbI{\mathbb I}

\newcommand{\Z}{\mathbf Z}
\newcommand{\Q}{\mathbf Q}
\newcommand{\R}{\mathbf R}
\newcommand{\C}{\mathbf C}
\newcommand{\A}{\mathbf A}    

\def\bbH{\mathbb H}

\def\frakp{{\mathfrak p}}
\def\frakP{\mathfak P}
\def\frakq{\mathfrak q}

\def\frakl{\mathfrak l}
\def\frakP{\mathfrak P}

\def\frakH{{\mathfrak H}}

\def\frakX{\mathfrak X}

\def\frakN{\mathfrak N}

\def\bfone{{\mathbf 1}}

\def\etale{{\'{e}tale }}

\newcommand{\<}{\langle}   
\renewcommand{\>}{\rangle} 

\def\ot{\otimes}

\def\hookto{\hookrightarrow}

\def\ol{\overline}  \nc{\opp}{\mathrm{opp}} \nc{\ul}{\underline}

\newcommand{\pair}[2]{\< #1, #2\>}

\def\XYmatrix{\xymatrix@M=8pt} 
\def\ncmd{\newcommand}
\ncmd{\xysubset}[1][r]{\ar@<-2.5pt>@{^(-}[#1]\ar@<2.5pt>@{_(-}[#1]}
\ncmd{\XYmatrixc}[1]{\vcenter{\XYmatrix{#1}}}
\ncmd{\xyto}[1][r]{\ar@{->}[#1]}
\ncmd{\xyinj}[1][r]{\ar@{^(->}[#1]}
\ncmd{\xysurj}[1][r]{\ar@{->>}[#1]}
\ncmd{\xyline}[1][r]{\ar@{-}[#1]}
\ncmd{\xydotsto}[1][r]{\ar@{.>}[#1]}
\ncmd{\xydots}[1][r]{\ar@{.}[#1]}
\ncmd{\xyleadsto}[1][r]{\ar@{~>}[#1]}
\ncmd{\xyeq}[1][r]{\ar@{=}[#1]} \ncmd{\xyequal}[1][r]{\ar@{=}[#1]}
\ncmd{\xyequals}[1][r]{\ar@{=}[#1]}
\ncmd{\xymapsto}[1][r]{l\ar@{|->}[#1]}\ncmd{\xyimplies}[1][r]{\ar@{=>}[#1]}
\ncmd{\xyiso}{\ar[r]_-{\sim}}
\def\injxy{\ar@{^(->}}

\newcommand{\pMX}[4]{\begin{pmatrix}
{#1}& {#2}\\
{#3}&{#4}\end{pmatrix} }

 \newcommand{\pDII}[2]{\begin{pmatrix}{#1}&0
 \\0&{#2}\end{pmatrix}}

\newcommand{\seesaw}[4]{{#1}\ar@{-}[rd]\ar@{-}[d]&{#2}\ar@{-}[d]\\
{#3}\ar@{-}[ru]&{#4}}

\def\ie{i.e. }

\def\cf{\mbox{{\it cf.} }}

\def\uf{\varpi} 

\def\ndivides{\nmid}

\def\x{{\times}}

\def\al{\alpha}

\def\Lam{\Lambda}

\def\om{\omega}

\def\iso{\simeq}
\def\con{\equiv}
\def\bksl{\backslash}
\newcommand\stt[1]{\left\{#1\right\}}
\def\ep{\epsilon}

\def\lam{\lambda}

\def\disjoint{\sqcup}

\renewcommand\pmod[1]{\,(\mbox{mod }{#1})}

\renewcommand\Re{\text{Re}\,}

\usepackage{hyperref}

\DeclareMathAlphabet{\mathpzc}{OT1}{pzc}{m}{it}
\allowdisplaybreaks[1]

\theoremstyle{definition}
\numberwithin{equation}{section}

\newtheorem{lem}[thm]{Lemma}
\newtheorem{rem}[thm]{Remark}

\renewcommand{\bar}{\overline}
\def\rt{{}^{\rm t}}

\def\rmd{{\rm d}}
\def\newform{\bff^\circ}

\title[Bessel period and the non-vanishing of Yoshida lifts]{Bessel periods and the non-vanishing of Yoshida lifts modulo a prime}
\author[M.L. Hsieh]{Ming-Lun Hsieh}
\address[Hsieh]{ 
Institute of Mathematics, Academia Sinica~\\ Taipei 10617, Taiwan\and National Center for Theoretic Sciences\and  
Department of Mathematics, National Taiwan University~
}
\email{mlhsieh@math.sinica.edu.tw}

\author[K. Namikawa]{Kenichi Namikawa}
\date{\today}
\address[Namikawa]{ Department of Mathematics, School of Engineering, Tokyo Denki University~\\
5, Asahicho, Senju, Adachi city, Tokyo, 120-8551, Japan~
}
\email{namikawa@mail.dendai.ac.jp}

\thanks{M.-L. Hsieh is partially supported by MOST grant 103-2115-M-002-012-MY5. K. Namikawa was supported by JSPS
Grant-in-Aid for Research Activity Start-up Grant Number 1566157.}
\begin{document}
\begin{abstract}
We give an explicit construction of vector-valued Yoshida lifts and derive a formula of the Bessel periods of Yoshida lifts, by which we prove the non-vanishing modulo a prime of Yoshida lifts attached to a pair of elliptic modular newforms. As a consequence, we obtain a new proof of the non-vanishing of Yoshida lifts.
\end{abstract}

\maketitle
\tableofcontents

\section{Introduction}

 In \cite{yo80} and \cite{yo84}, Yoshida constructed certain explicit scalar-valued Siegel modular forms associated with a pair of elliptic modular newforms (case (I)) or a Hilbert modular newform (case (II)). These modular forms, known as Yoshida lifts, are theta lifts from ${\rm O}_{4,0}$ to ${\rm Sp}_4$. Yoshida conjectured the non-vanishing of these theta lifts under certain assumptions, and the non-vanishing of Yoshida lifts in case (I) was later proved by B\"{o}cherer and Schulze-Pillot in \cite{bsp91} and \cite{bsp97} (see also \cite{brooks01} for the representation technique). The purpose of this paper is to (i) extend Yoshida's construction to Siegel modular forms valued in $\Sym^{2k_2}(\C^{\oplus 2})\ot\det^{k_1-k_2+2}$ and calculate their Bessel periods; (ii) show the non-vanishing of Yoshida lifts modulo a prime $\ell$ under some mild conditions in case (I). In particular, we obtain a new proof of  the non-vanishing of Yoshida lifts in this case. 

To state our main results explicitly, we introduce some notation. Let $N^-$ be a square-free product of an odd number of primes and $(N^+_1,N^+_2)$ be a pair of positive integers prime to $N^-$. Put $(N_1,N_2):=(N^-N^+_1,N^-N^+_2)$. Let $(f_1,f_2)$ be a pair of elliptic modular newforms of level $(\Gamma_0(N_1),\Gamma_0(N_2))$ and weight $(2k_1+2,2k_2+2)$. Assume that $k_1\geq k_2\geq 0$.
Let $D$ be the definite quaternion algebra of absolute discriminant $N^-$. By the Jacquet-Langlands-Shimizu correspondence, to each $f_i$ ($i=1,2$), we can associate a vector-valued newform $\bff_i:D^\times\backslash D^\times_{\mathbf A} \to {\rm Sym}^{2k_i}(\C^{\oplus 2})\otimes{\rm det}^{-k_i}$ on $D^\times_{\mathbf A}$ unique up to scalar such that $\bff_i$ shares the same Hecke eigenvalues with $f_i$ at all $p\nmid N^-$. Thus $(\bff_1,\bff_2)$ gives rise to a vector-valued automorphic form $\bff_1\ot\bff_2$ on $\GSO(D)$. Combined with an appropriate (vector-valued) Bruhat-Schwartz  function $\varphi$ on $D^{\oplus 2}_{\mathbf A}$ (See \subsecref{SS:testfunctions}), one obtains Yoshida lift $\theta^\ast_{\bff_1\ot\bff_2}$ by global theta lifts from $\GSO(D)$ to $\Sp_4$. This Yoshida lift $\theta^\ast_{\bff_1\ot\bff_2}$ is a degree two holomorphic Siegel modular form of weight $\Sym^{2k_2}\ot\det^{k_1-k_2+2}$ and level $\Gamma_0^{(2)}(N)$ with $N={\rm l.c.m.}(N_1,N_2)$, and moreover, it is also a Hecke eigenform with the spin $L$-function $L(\theta_{\bff_1\ot\bff_2},s)=L(f_1,s-k_2)L(f_2,s-k_1)$. For each prime factor $p$ of ${\rm g.c.d.}(N_1,N_2)$, we denote by $\epsilon_p(f_1), \epsilon_p(f_2)\in \{\pm 1\}$ the Atkin-Lehner eigenvalues at $p$ on $f_1$ and $f_2$ respectively. We consider the following condition which is necessary for the non-vanishing of Yoshida lifts (\cf \cite[Lemma\,4.2]{yo84}). \beqcd{LR}\text{$\epsilon_p(f_1)=\epsilon_p(f_2)$ for every prime $p$ with ${\rm ord}_p(N_1)={\rm ord}_p(N_2)>0$.}\eeqcd

Let $\ell$ be a rational prime and fix a place $\lambda$ of $\bar{\Q}$ above $\ell$. Then it is known that one can normalize forms $\bff_1,\bff_2$ on $D_\A$ so that the values of $\bff_i$ on the finite part $\wh D^\x$ are $\lam$-integral and do not completely vanish modulo $\lam$ (See \subsecref{SS:normalizationY}). Our main result is about the non-vanishing of $\theta^*_{\bff_1\ot\bff_2}$ modulo $\lam$ attached to this normalized $\bff_1\ot\bff_2$.
\begin{thmA}[\thmref{T:nonvanYoshida}]\label{thma}
Assume that \eqref{LR} holds and the prime $\ell$ satisfies the following conditions
\begin{enumerate}
\item\label{tma(1)} $\ell > 2k_1$ and $\ell\nmid 2N$
 \item\label{tma(2)} the residual Galois representations $\bar{\rho}_{f_i,\ell}:{\rm Gal}(\bar{{\mathbf Q}}/{\mathbf Q})\to {\rm GL}_2(\bar{\mathbf F}_\ell)$ attached to $f_i$ are absolutely irreducible.
\end{enumerate}
 Then the Yoshida lift $\theta_{\bff_1\ot\bff_2}^*$ has $\lambda$-integral Fourier expansion, and there are infinitely many Fourier coefficients which are nonzero modulo $\lambda$.
\end{thmA}
It is well-known that the conditions (i) and (ii) only exclude finitely many primes $\ell$ (\cf\cite[Proposition\,3.1]{dimitrov05}), so we obtain immediately a new proof of the non-vanishing of Yoshida lifts in case (I) from \thmref{thma}.
\begin{corb*}Suppose that \eqref{LR} holds. Then the Yoshida lift $\theta_{\bff_1\ot\bff_2}^*$ is nonzero.
\end{corb*}
When $N_1$ and $N_2$ are square-free, the nonvanishing of Yoshida lifts in case (I) has been proved in \cite{bsp97} by a completely different method.\\

Our main motivation for the study of the non-vanishing modulo $\lam$ of Yoshida lifts in case (I) originates from the applications to the Bloch-Kato conjecture for the special value of Rankin-Selberg $L$-functions $L(f_1\ot f_2,s)$ at $s=k_1+k_2+2$. For example, the authors in \cite{ak13} and \cite{bdsp12} use the method of \emph{Yoshida congruence} to construct non-trivial elements in the Bloch-Kato Selmer group associated with the four dimensional $\ell$-adic Galois representation $\rho_{f_1,\ell}\otimes \rho_{f_2,\ell}(-k_1-k_2-1)$. Roughly speaking, under some strong hypotheses these authors show that if $\ell$ divides the algebraic part of $L(f_1\ot f_2,k_1+k_2+2)$, then $\ell$ is a congruence prime for the Yoshida lift $\theta_{\bff_1\ot\bff_2^*}$, and hence for such primes, they can construct non-trivial congruences between Hecke eigen-systems of Yoshida lifts $\theta_{\bff_1\ot\bff_2}^*$ and stable forms on ${\rm GSp}_4$. This in turn gives rise to the congruences between their associated Galois representations, with which the authors can construct elements in the desired Bloch-Kato Selmer groups by adapting the method in \cite{brown07} for the case of Saito-Kurokawa lifts. The non-vanishing modulo $\lam$ of explicit Yoshida lifts serves as the first step in the method of Yoshida congruence (\cite[Corollary 9.2]{bdsp12} and \cite[Theorem 6.5]{ak13}). In \cite{ak13}, the authors use a result of Jia \cite{ji10} on the non-vanishing modulo $\ell$ of scalar-valued Yoshida lifts (\ie $k_2=0$), which is conditional under the assumption of Artin's primitive root conjecture. Our \thmref{thma} relaxes the assumption of Artin's conjecture and further extends Jia's result to vector-valued Yoshida lifts.

The proof of \thmref{thma} is based on an explicit Bessel period formula for Yoshida lifts (\propref{Besselper}), where we prove that the Bessel period of Yoshida lifts associated to a ring class character $\phi$ of an imaginary quadratic field $K$ is actually a product of a local constant $e(\bff_1\ot\bff_2,\phi)$ defined in \eqref{E:localconstant} and the toric period integrals attached to $\bff_1\ot\phi$ and $\bff_2\ot\phi^{-1}$ over $K$. On the other hand, it is shown that Bessel periods, after a suitable normalizaton, can be written as a linear combination of Fourier coefficients of Yoshida lifts $\theta_{\bff_1\ot\bff_2}^*$ (See \eqref{E:BesselFC} in \lmref{L:BesselFC}). Therefore, the non-vanishing of $\theta_{\bff_1\ot\bff_2}^*$ modulo $\lam$ boils down to the non-vanishing modulo $\lam$ of the local constant $e(\bff_1\ot\bff_2,\phi)$ and toric period integrals of $\bff_1\ot\phi$ and $\bff_2\ot\phi^{-1}$ for some ring class character $\phi$. Finally, we prove that if $\phi$ is sufficiently ramified, then the assumption \eqref{LR} implies the non-vanishing of the local constant $e(\bff_1\ot\bff_2,\phi)$, and the simultaneous non-vanishing modulo $\lam$ of these toric integral periods is a direct consequence of results of Masataka Chida and the first author in \cite{ch15}. 

So far we focus on Yoshida lifts in case (I). Let us make a remark on the non-vanishing modulo $\lam$ of Yoshida lifts in case (II), \ie theta lifts attached to Hilbert modular newforms $f$ over a real quadratic field $F$. We also give an explicit construction of Yoshida lifts in case (II) and show that their Bessel period formula attached to a ring class character $\phi$ of an imaginary quadratic field $K$ is a product of a local constant and a toric period integral attached to $f$ and the character $\phi\circ\rmN_{E/F}$ over $E:=FK$. However, it is not clear to us how to show the non-vanishing modulo $\lam$ of this toric period integral for sufficiently ramified $\phi$ despite that the main results in \cite{ch15} have been extended to Hilbert modular forms by P.-C. Hung \cite{hu14}. We hope to come back to this case in the future. 

This paper is organized as follows.
After introducing some basic notation in \subsecref{secnot}, 
we give the construction of Yoshida lifts in \secref{secYos}.
The particular choice of Bruhat-Schwartz function $\varphi$ is made in \subsecref{SS:testfunctions} and the Fourier coefficients of Yoshida lifts are given by \propref{P:FCYoshida}. We calculate the Bessel periods of Yoshida lifts in \secref{S:Bessel}, and the Bessel period formula is given in \propref{Besselper}. Finally, we prove the non-vanishing modulo $\lambda$ of Yoshida lifts in \secref{secFC}.

\section{Notation and definitions}\label{secnot}

\subsection{}
If $v$ is a place of $\Q$, we let $\Q_v$ be the completion of ${\mathbf Q}$ at $v$ and $|\cdot|_v$ be the normalized absolute value on $\Q_v$. Let $\wh\Z$ be the finite completion of $\Z$. If $M$ is an abelain group, let $M_v=M\ot_\Z\Q_v$ and $\wh M=M\ot_\Z\wh\Z$. Let $\A=\R\x\wh\Q$ be the ring of adeles of $\Q$ and $\A_f=\wh\Q$ be the finite adeles of $\Q$. If $F$ is an \etale algebra over $\Q$ (or $\Qp$), denote by ${\mathcal O}_F$ the ring of integers of $F$ and by $\Delta_F$ the absolute discriminant of $F$.

If $G$ is an algebraic group $G$ over $\Q$, denote by $Z_G$ the center of $G$. If $R$ is a $\Q$-algebra, denote by $G(R)$ the group of $R$-rational points of $G$. If $g\in G(\A)$, we write $g_f\in G(\A_f)$ for the finite component of $g$ and $g_v\in G(\Q_v)$ for its $v$-component. We sometimes write $G_\Q=G(\Q)$ and $G_{\mathbf A}=G(\A)$ for brevity. Define the quotient space $[G]$ by \[[G]:=G_\Q\backslash G_{\mathbf A}.\] If $\rmd g$ is a Haar measure on $G_\A$, then the quotient space $G_\Q\bksl G_\A$ is equipped with the quotient measure of $\rmd g$ by the counting measure of $G_\Q$, which we shall still denote by $\rmd g$ if no confusion arises.

For a set $S$, $\sharp(S)$ denotes the cardinality of $S$ and ${\mathbb I}_S$ denotes the characteristic function of $S$.

\subsection{Algebraic representation of ${\rm GL}_2$}\label{algrep}
Let $A$ be an $\Z$-algebra. We let $A[X,Y]_n$ denote the space of two variable homogeneous polynomial of degree $n$ over $A$.
Suppose $n!$ is invertible in $A$.
We define the perfect pairing $\langle\cdot ,\cdot \rangle_n:A[X,Y]_n\times A[X,Y]_n\to A$ by
\begin{align*}
  \langle X^iY^{n-i}, X^jY^{n-j} \rangle_n
 =\begin{cases}  (-1)^i\binom{n}{i}^{-1},  & {\rm if } j+i=n, \\
                       0, & {\rm if } i+j\neq n. \end{cases} 
\end{align*}
For $\kappa = (n+b,b)\in \Z^2$ with $n\in \Z_{\geq 0}$, let $\cL_\kappa(A)$ denote $\Sym^{n}(A^{\oplus 2})\ot\det^b$ the algebraic representation  of ${\rm GL}_2(A)$
with the highest weight $\kappa$. In other words, ${\mathcal L}_\kappa(A)=A[X,Y]_n$ with $\rho_\kappa:{\rm GL}_2(A) \to {\rm Aut}_A{\mathcal L}_\kappa(A)$given by
\begin{align*}
  \rho_\kappa(g) P(X,Y) = P((X,Y)g) \cdot (\det g)^b.  
\end{align*}
It is well-known that the pairing $\langle\cdot,\cdot\rangle_n$ on ${\mathcal L}_\kappa(A)$ satisfies
\begin{align*}
     \langle \rho_\kappa(g)v, \rho_\kappa(g)w \rangle_n
  = (\det g)^{n+2b} \cdot \langle v,w \rangle_n \quad (g\in {\rm GL}_2(A)).
\end{align*}
For each non-negative integer $k$, we put
\begin{align*}
  ({\mathcal W}_k(A), \tau_k) := (A[X,Y]_{2k}, \rho_{(k,-k)}). 
\end{align*}
Then $({\mathcal W}_k(A), \tau_k)$ is the algebraic representation of ${\rm PGL}_2(A) = {\rm GL}_2(A)/A^\times$,
and the pairing $\langle\cdot ,\cdot \rangle_{2k}$ is ${\rm GL}_2(A)$-equivariant.


\subsection{Siegel modular forms of degree two and Fourier expansions}\label{secBessel}
Let ${\rm GSp}_4$ be the algebraic group defined by
\begin{align*}
    {\rm GSp}_4 
 = \left\{  g\in {\rm GL}_4  | 
              g\begin{pmatrix} 0 & \bfone_2 \\ -\bfone_2 & 0 \end{pmatrix} {}^{\rm t} g 
              = \nu(g) \begin{pmatrix} 0 & \bfone_2 \\ -\bfone_2 & 0 \end{pmatrix}  
                 \right\}  
\end{align*}
with the similitude character $\nu: {\rm GSp}_4\to {\mathbb G}_m$. Here $\bfone_2$ denotes the $2$ by $2$ identity matrix. The Siegel upper half plane of degree 2 is defined by 
\[\frakH_2=\stt{Z\in {\rm M}_2(\C)\mid Z=\rt Z,\,\Im Z\text{ is positive definite}}.\]
Then $\frakH_2$ is equipped with an action of $\Sp_4(\R)$ given by $g\cdot Z=(AZ+B)(CZ+D)^{-1}$ for $g=\pMX{A}{B}{C}{D}$ and $Z\in\frakH_2$, and 
define the automorphy factor $J:\Sp_4(\R)\times{\mathfrak H}_2 \to {\rm GL}_2(\C)$ by
$J(g,Z) = CZ+D$.
Let ${\mathbf i}:=\sqrt{-1}\cdot \bfone_2\in {\mathfrak H}_2$. 
Let $\bfK_\infty$ be the maximal compact subgroup of ${\rm GSp}_4(\R)$ defined by
\begin{align*}
  \bfK_\infty = \{ g\in {\rm GSp}_4(\R) | g{}^{\rm t}g = \bfone_2 \}. 
\end{align*}
Let $\kappa=(a,b)\in \Z^2$ with $a-b\in 2\Z_{\geq 0}$. For a positive integer $N$, let
\[U^{(2)}_0(N)=\stt{g=\pMX{A}{B}{C}{D}\in \GSp_4(\wh\Z)\mid A,B,C,D\in{\rm M}_2(\wh\Z),\,C\con 0\pmod{N}}.\]
be an open-compact subgroup of ${\rm GSp}_4({\mathbf A}_f)$. For each quadratic character $\chi:\Q^\x\bksl \A^\x\to\stt{\pm 1}$, denote by ${\mathcal A}_\kappa({\rm GSp}_4({\mathbf A}),N,\chi)$ the space of adelic Siegel modular forms of weight $\kappa$, level $N$ and type $\chi$, which
consists of smooth functions 
$\cF:{\rm GSp}_4({\mathbf A})\to {\mathcal L}_\kappa(\C)$ such that
\begin{align*}
   \cF( \gamma g k_\infty u z )
 = &\rho_\kappa(J(k_\infty,{\mathbf i})^{-1} ) \cF(g) \chi(\det D), \\
( \gamma\in\GSp_4(\Q),\,k_\infty\in \bfK_\infty,\,&
 u=\pMX{A}{B}{C}{D}\in U^{(2)}_0(N),\,z\in\A^\x).
\end{align*}

\subsubsection*{Fourier coefficients of $\cF$}
Denote by $\cH_2$ the group of $2$ by $2$ symmetric matrices. 
Let $U$ be a unipotent subgroup of $\GSp_4$ defined by 
\begin{align*}
 U = \left\{ u(X) = \begin{pmatrix} \bfone_2 & X \\ 0 & \bfone_2  \end{pmatrix} 
                 \mid X\in \cH_2   \right\}.
\end{align*}
Let $\psi=\prod_v\psi_v:\A/\Q\to \C^\times$ be the additive character 
with $\psi(x_\infty)=\exp(2\pi\sqrt{-1}x_\infty)$ for $x_\infty\in \R=\Q_\infty$.
For each $S\in \cH_2(\Q)$, 
let $\psi_S: U_\Q\backslash U_{\mathbf A} \to \C^\times$ be the additive character defined by
$\psi_S(u(X))=\psi({\rm Tr}(-SX))$.
The adelic $S$-th Fourier coefficient ${\bfW}_{\cF,S}:{\rm GSp}_4({\mathbf A}) \to {\mathcal L}_\kappa(\C)$ is defined by
\begin{align*}
  {\bfW}_{\cF,S} (g) =\int_{U_\Q\backslash U_{\mathbf A}} \cF(ug)\psi_S(u) \rmd u,  
\end{align*}
where $\rmd u$ is the Haar measure with $\vol(U_\Q
\bksl U_\A,\rmd u)=1$. Then $\cF$ has the Fourier expansion
\beq\label{E:FC} \cF(g) = \sum_{S\in \cH_2(
\Q) } {\bfW}_{\cF,S}(g).  
\eeq
Note that $\bfW_{\cF,S}(ug)=\psi_S(u)\bfW_{\cF,S}(g)$ and \beq\label{E:FC3}\bfW_{\cF,S}(\pDII{\xi}{\nu{}^{\rm t}\xi^{-1}}g)=\bfW_{\cF,\nu\rt\xi S\xi}(g)\eeq for $\xi\in\GL_2(\Q)$ and $\nu\in \Q^\x$. 

\section{Yoshida lifts}\label{secYos}
\subsection{Orthogonal groups}\label{secYos1}

Let $D_0$ be a definite quaternion algebra over $\Q$ of discriminant $N^-$ and let $F$ be a quadratic \etale algebra over $\Q$. Let $D = D_0\otimes_\Q F$. We assume that every place dividing $\infty N^-$ is split in $F$. It follows that $F$ is either $\Q\oplus \Q$ or a real quadratic field over $\Q$, and $D$ is precisely ramified at $\infty N^-$. Denote by $x\mapsto x^\ast$ the main involution of $D_0$ and by $x\mapsto \ol{x}$ the non-trivial automorphism of $F/\Q$, which are extended to automorphisms of $D$ naturally. We define the four dimensional quadratic space $(V,{\rm n})$ over $\Q$ by \[
  V = \left\{ x\in D : \ol{x}^* = x   \right\},\,{\rm n}(x)=xx^*.\]
  Let $H$ be the algebraic group over $\Q$ given by 
\[H (\Q)=D^\times \times_{F^\times} \Q^\times = D^\times \times \Q^\times /\{(a, {\rm N}_{F/\Q}(a)): a\in F^\times \}. \]Then $H$ acts on $V$ via $\varrho: H \to\Aut V$ given by
\[   \varrho(a,\alpha)(x) = \alpha^{-1} a x \ol{a}^*\quad (x, \in V, (a,\alpha)\in B^\times).  \label{varrho}   
\]
This induces an identification $\varrho: H\iso \GSO(V)$ with the similitude map given by \[\nu(\rho(a,\alpha))=\al^{-2}{\rm N}_{F/\Q}(aa^*). \]
For $a\in D_\A^\x$, we write $\varrho(a)=\varrho(a,1)$. Put
\[H^{(1)}=\stt{h\in H\mid \nu(\varrho(h))=1}\iso \SO(V).\]
\begin{remark}
If $v=w\ol{w}$ is a place split in $F$, then $F=\Q_v e_w\oplus \Q_v e_{\ol{w}}$, where $e_w$ and $e_{\ol{w}}$ are idempotents corresponding to $w$ and $\ol{w}$ respectively, and each place $w$ lying above $v$ induces the isomorphisms \beq\label{E:splittrivialization}\begin{aligned}i_w:D_{0,v}^\x\x D_{0,v}^\x/\Q_v^\x\iso H(\Q_v),&\quad (a,d)\mapsto (ae_w+de_{\ol{w}},{\rm n}(d));\\ j_w: D_{0,v}\iso V_v,&\quad  x\mapsto xe_w+x^*e_{\ol{w}}.\end{aligned}\eeq
By definition, $\varrho(i_w(a,d))j_w(x)=j_w(axd^{-1})$.
\end{remark}
\def\bfx{x}

\subsection{Notation for quaternion algebras}\label{SS:DATA}
We will fix the following data throughout this paper. 
For any ring $A$, the main involution $*$ on ${\rm M_2}(A)$ is given by \[\pMX{a}{b}{c}{d}^*=\pMX{d}{-b}{-c}{a}.\]
Fix an isomorphism $\Phi=\prod_{p\ndivides N^-\infty}\Phi_p: \prod_{p\ndivides N^-}'{\rm M}_2(\Qp)\iso \prod'_{p\ndivides N^-}D_0\ot\Qp$ once and for all. Let $\cO_{D_0}$ be the maximal order of $D_0$ such that $\cO_{D_0}\ot\Zp=\Phi_p({\rm M}_2(\Zp))$ for all $p\ndivides N^-$ and let $\cO_D:=\cO_{D_0}\ot_\Z\cO_F$ be a maximal order of $D$.

Let $N^+$ be a positive integer with$(N^+,\Delta_FN^-)=1$ and let $R$ be the standard Eichler orders of $D$ of level $N^+\cO_F$ contained in $\cO_D$. Then the algebraic group $H$ and the quadratic space $V$ can be endowed with an integral structure induced by $R$ as follows. Define the lattice \[V(\Z):=V\cap \wh R\,;\quad V(A):=V(\Z)\ot_{\Z} A\]for any ring $A$.
Define an open-compact subgroups $H(\wh\Z)$ and $\cU$ by 
\beq\label{E:opcpt} 
\begin{aligned}
H(\wh\Z)=&\prod_p H(\Zp),\quad H(\Zp):=R_p^\x\x_{\cO_{F_p}^\x}\Zp^\x;\\
\cU=&H^{(1)}(\A_f)\cap H(\wh\Z)=\stt{(h,\alpha)\in H^{(1)}(\A_f)\mid h\in \wh R^\x}.
\end{aligned}\eeq

Define the quaternion algebra $\bbH_\Q$ over $\Q$ by
\begin{align*}
  \bbH_\Q = \stt{\pMX{z}{w}{-\ol{w}}{\ol{z}}\mid z,w\in\Q(\sqrt{-1})}.\end{align*} 
The main involution $\ast:{\mathbb H}_\Q\to {\mathbb H}_\Q$ is given by $x\mapsto {}^{\rm t}\bar{x}$ and let $\cO_{\bbH_\Q}$ be the maximal order defined by 
\[\cO_{\bbH_\Q}=\Z+\Z i+\Z j+\Z \frac{1+i+j+ij}{2},\]
where $i=\pDII{\sqrt{-1}}{-\sqrt{-1}}$ and $j=\pMX{0}{\sqrt{-1}}{\sqrt{-1}}{0}$. Then the Hamilton quaternion algebra $\bbH:=\bbH_\Q\ot\R$.
Let $\ell\ndivides 2N^-$ be a prime. For the later study on the $\ell$-integrality of Yoshida lifts in \secref{secFC}, we take $\Phi_\infty: \bbH \iso D_{0,\infty}$ to be an isomorphism compatible with this prime $\ell$ in the following manner. Choose a real quadratic field $F_1$ such that $\ell$ is split in $F_1$ and every prime factor of $2N^-$ is inert in $F_1$. Fix an embedding $F_1\hookto \Q_\ell$. Then there is an isomorphism $\Phi_{F_1}: \bbH_{\Q}\ot F_1\iso D_0\ot F_1$ such that $\Phi_{F_1}(\cO_{\bbH_\Q}\ot\Z_\ell)=\cO_{D_0}\ot\Z_\ell$, and the isomorphism $\Phi_\infty: \bbH \iso D_{0,\infty}$ is obtained by extending $\Phi_{F_1}$ by scalars. 

Let $F':=FF_1(\sqrt{-1})$. For any $F'$-algebra $L$, $\Phi_{F_1}^{-1}$ induces an embedding $D_0\hookto \bbH_\Q\ot_\Q L\hookto {\rm M}_2(L)$, which in turn induces $D^\x\hookto \GL_2(L\ot F)=\GL_2(L)\x\GL_2(L)$. Therefore, for each pair of non-negative integers $(k_1,k_2)$, we can regard $(\tau_{k_1}\ot \tau_{k_2}, \cW_{k_1}\ot\cW_{k_2})$ in \subsecref{algrep} as an algebraic representation of $D^\x/F^\x$ ($=H/Z_H)$ over $L$.
\subsection{Automorphic forms on $H(\A)$}\label{subsec:fromsonH}
Let $\ul{k}=(k_1,k_2)$ be a pair of positive integers with $k_1\geq k_2$ and let $(\tau_{\ul{k}},\cW_{\ul{k}}):=(\tau_{k_1}\otimes\tau_{k_2},{\mathcal W}_{k_1}\otimes {\mathcal W}_{k_2})$ be an algebraic representation of $D^\x$. 
For any open-compact subgroup $U\subset \wh \cO_D^\x$, denote by ${\mathcal A}_{\ul{k}}(D^\times_{\mathbf A},U)$ the space of modular forms on $D_\A^\x$ of weight $\ul{k}$, consisting of functions
$\bff:D^\times_{\mathbf A}\to {\mathcal W}_{\ul{k}}(\C)$
such that
\begin{align*}
 &\bff(z\gamma h u) = \tau_{\ul{k}}(h^{-1}_\infty) \bff(h_f),\\
 (h=&(h_\infty,h_f)\in D^\times_{\mathbf A},\, (z,\gamma, u)\in F^\times_{\mathbf A}\times D^\times \times U).    
\end{align*}
Hereafter, we shall view $\bff$ as an automorphic form on $Z_{H}(\A)\bksl H(\A)$ by the rule $\bff(a,\alpha):=\bff(a)$.

Let $\frakN^+\mid N^+$ be an ideal $\cO_F$ and let $\frakN=\frakN^+N^-$. Let $R_{\frakN^+}$ be the Eichler of level $\frakN^+$ contained in $\cO_D$ (so $R\subset R_{\frakN^+}$). Let $\cA(D_\A^\x)$ be the space of automorphic forms on $D_\A^\x$. Then there is a natural identification \[\cA_{\ul{k}}(D_\A^\x,\wh R_{\frakN^+}^\x)=\Hom_{D_\infty^\x}(\cW_{\ul{k}}(\C),\cA(D_\A^\x)^{\wh R_{\frakN^+}^\x}).\]
Let $f^{\rm new}$ be a newform on $\PGL_2(F_\A)$ of weight $2\ul{k}+2=(2k_1+2,2k_2+2)$ and level $\frakN$. Namely, $f^{\rm new}$ is a pair of elliptic modular newforms $(f_1,f_2)$ of level $(\Gamma_0(N_1^+N^-)$, $\Gamma_0(N_2^+N^-))$ 
and weight $(2k_1+2, 2k_2+2)$ if $F=\Q\oplus \Q$ and $\frakN^+=(N_1^+,N_2^+)$, while $f^{\rm new}$ is a Hilbert modular newform of level $\Gamma_0(\frakN)$ and weight $2\ul{k}+2$ if $F$ is a real quadratic field. Let $\pi$ be the automorphic cuspidal representation of $\PGL_2(F_\A)$ attached to $f^{\rm new}$ and let $\pi^D\subset\cA(D_\A^\x)$ be the Jacquet-Langlands transfer of $\pi$, which is an automorphic representation of $D_\A^\x$. Then the subspace $\cA_{\ul{k}}(D_\A^\x,\wh R_{\frakN^+}^\x)[\pi^D]:=\Hom_{D_\infty}(\cW_{\ul{k}}(\C),(\pi^D)^{\wh R_{\frakN^+}^\x})$ has one-dimensional by the theory of newforms. Any generator $\newform$ of this space $\cA_{\ul{k}}(D_\A^\x,\wh R_{\frakN^+}^\x)[\pi^D]$ shall be called the newform associated with $f^{\rm new}$.
\subsection{Weil representation on ${\rm O}(V)\times{\rm Sp}_4$} \label{secWeil}
Let $(\cdot, \cdot): V\times V \to \Q$ be the bilinear form defined by
$(x,y) = {\rm n}(x+y) - {\rm n}(x) - {\rm n}(y)$. 
Denote by ${\rm GO}(V)$ the orthogonal similitude group with the similitude morphism 
$\nu:{\rm GO}(V) \to {\mathbb G}_m$.
Let ${\mathbf X} = V\oplus V$. For $v$ a place of $\Q$, let $V_v=V\otimes_\Q\Q_v$ and ${\mathbf X}_v={\mathbf X}\otimes_\Q\Q_v$.  Note that the quadratic character $\chi_{F_v/\Q_v}$ attached to $F_v/\Q_v$ is the quadratic character attached to $V_v$. Denote by ${\mathcal S}({\mathbf X}_v)$ the space of $\C$-valued Bruhat-Schwartz functions on ${\mathbf X}_v$. For each $\bfx=(x_1,x_2)\in {\mathbf X}_v=V_v\oplus V_v$, we put
\begin{align*}
  S_\bfx= \begin{pmatrix} {\rm n}(x_1) & \frac{1}{2}(x_1,x_2)  \\
                                \frac{1}{2}(x_1,x_2) & {\rm n}(x_2)   \end{pmatrix}. 
\end{align*}
Let $\omega_{V_v}:{\rm Sp}_4(\Q_v)\to {\rm Aut}_\C{\mathcal S}({\mathbf X}_v)$ be the Schr\"odinger realization of the Weil representation. For every $\varphi\in {\mathcal S}({\mathbf X}_v)$, we have
\begin{align*}
\omega_{V_v} \left(  \begin{pmatrix} A & 0 \\ 0 & {}^{\rm t}A^{-1} \end{pmatrix} \right) \varphi(x)
 = & \chi_{F_v/\Q_v}(\det A ) |\det A|^2_p\cdot \varphi(xA),  \\
\omega_{V_v} \left(  \begin{pmatrix} \bfone_2 & B \\ 0 & \bfone_2 \end{pmatrix} \right) \varphi(x)
 = & \psi_v({\rm Tr}(S_xB))\cdot \varphi(x),  \\
\omega_{V_v} \left(  \begin{pmatrix} 0 & \bfone_2 \\ -\bfone_2 & 0 \end{pmatrix} \right) \varphi(x)
 = & \gamma^2_{V_v} \cdot \widehat{\varphi}(x),  
\end{align*}
where $\gamma_{V_v}=\gamma(\psi_v\circ{\rm n})$ is the Weil index attached to the second degree character $\psi_v\circ{\rm n}:V_v\to\C^\x$ (\cf\cite[Theorem A.1]{rao93}), and $\widehat{\varphi}\in\cS(\mathbf X_v)$ is the Fourier transform of $\varphi$
with respect to the self-dual Haar measure $\rmd\mu$ on $V_v\oplus V_v$ defined by
\begin{align*}
 \widehat{\varphi}(x) := \int_{{\mathbf X}_v} \varphi(y) \psi_v((x,y))\rmd\mu(y).  
\end{align*}
Let $\cR(\GO(V)\x \GSp_4)$ be the $R$-group
\[\cR(\GO(V)\x \GSp_4)=\stt{(h,g)\in \GO(V)\x\GSp_4\mid \nu(h)=\nu(g)}.\]
Then the Weil representation can be extended to the $R$-group by
\begin{align*}
\omega_v:&\cR({\rm GO}(V_v)\times {\rm GSp}_4(\Q_v))\to {\rm Aut}_\C{\mathcal S}({\mathbf X}_v),\\
    \omega_v(h,g)\varphi(x)
  = &|\nu(h)|^{-2}_v(\omega_{V_v}(g_1)\varphi)(h^{-1}x)
\quad (g_1 = \begin{pmatrix} \bfone_2 & 0 \\ 0 & \nu(g)^{-1}\bfone_2 \end{pmatrix}g). 
\end{align*}
Let ${\mathcal S}({\mathbf X}_{\mathbf A}) = \otimes_v'{\mathcal S}({\mathbf X}_v)=\cS(\bfX_\infty)\ot\cS(\wh\bfX)$ ($\wh\bfX=\bfX\ot\wh\Z$). Define
$\omega_V = \otimes_v\omega_{V_v}: {\rm Sp}_4({\mathbf A}) \to {\rm Aut}_\C{\mathcal S}({\mathbf X}_{\mathbf A})$ 
and 
$\omega=\otimes_v\omega_v: \cR({\rm GO}(V)_{\mathbf A}\times {\rm GSp}_4({\mathbf A}))
\to {\rm Aut}_\C{\mathcal S}({\mathbf X}_{\mathbf A})$.

\subsection{Theta lifts}
Let $\bff\in\cA_{\ul{k}}(D^\x_\A,\wh R^\x)$. Define the pairing on $\cW_{\ul{k}}(\C)$ by 
$\pair{\cdot}{\cdot}_{2\ul{k}}=\langle\cdot,\cdot\rangle_{2k_1}\otimes\langle\cdot,\cdot\rangle_{2k_2}$, 
where $\langle\cdot,\cdot\rangle_{2k_i} (i=1,2)$ is the pairing introduced in Section \ref{algrep}.  Let \[\kappa=(k_1+k_2+2,k_1-k_2+2).\] For each vector-valued Bruhat-Schwartz function $\varphi\in{\mathcal S}({\mathbf X}_{\mathbf A})\otimes {\mathcal W}_{\ul{k}}(\C)\otimes{\mathcal L}_\kappa(\C)$, 
define the theta kernel $\theta(-,-;\varphi): \cR({\rm GO}(V)_{\mathbf A}\times {\rm GSp}_4({\mathbf A}) ) \to {\mathcal W}_{\ul{k}}(\C)\otimes{\mathcal L}_\kappa(\C)$ by
\begin{align*}
  \theta(h,g;\varphi) = \sum_{x\in{\mathbf X}} \omega(h,g)\varphi(x).
\end{align*}
Let $\GSp^+_4$ be the group of elements $g\in\GSp_4$ with $\nu(g)\in \nu(\GO(V))$. Define the theta lift $\theta(-; {\mathbf f},\varphi): \GSp_4^+(\Q)\bksl {\rm GSp}_4^+({\mathbf A})\to {\mathcal L}_\kappa(\C)$ by 
\begin{align*}
  \theta(g;{\mathbf f},\varphi) = \int_{[H^{(1)}]} \pair{\theta( hh^\prime, g;\varphi)}{{\mathbf f}(hh^\prime)}_{2\ul{k}}{\rm d}h \quad (\nu(h^\prime) = \nu(g)).
\end{align*}
Here $\rmd h:=\rmd h_\infty \rmd h_f$ is the Haar measure of $H^{(1)}(\A)$ normalized so that $\rmd h_\infty$ and $\rmd h_f$ are the Haar measures of $H^{(1)}(\R)$ and $H^{(1)}(\A_f)$ with 
  $\vol(H^{(1)}(\R),\rmd h_\infty)=\vol(H^{(1)}(\A_f)\cap \cU,\rmd h_f)=1$. Here $\cU$ is the group defined in \eqref{E:opcpt}. We extend uniquely $\theta(-;\bff,\varphi)$ to a function on $\GSp_4(\Q)\bksl \GSp_4(\A)$ by defining $\theta(g,\bff,\varphi)=0$ for $g\not\in\GSp_4(\Q)\GSp_4^+(\A)$.
  \subsection{The test functions}\label{SS:testfunctions}
Let $N=N^+N^-$ and $N_F={\rm l.c.m.}(N,\Delta_F)$. We choose a distinguished Bruhat-Schwartz function $\varphi=\varphi_\infty\ot\varphi_f\in\cS(\bfX_\A)\ot {\mathcal W}_{\underline{k}}(\C)\otimes {\mathcal L}_\kappa(\C)$ as follows. At the finite component, define $\varphi_f\in {\mathcal S}(\wh\bfX)$ by 
\beq\label{finitetest}
  \varphi_f =   {\mathbb I}_{ V(\wh\Z)\oplus V(\wh\Z)} \text{ the characteristic function of $V(\wh\Z)\oplus V(\wh\Z)$. }  
\eeq
\begin{lm}\label{Yoshidalev}
For $g=\begin{pmatrix} A & B \\ C & D \end{pmatrix}\in U^{(2)}_0(N_F)\cap\Sp_4(\widehat{\mathbf Z})$, we have \[\omega_V(g)\varphi_f = \chi_{F/\Q}({\rm det}D) \varphi_f,\]
where $\chi_{F/\Q}:\Q^\x\bksl \A^\x\to\stt{\pm 1}$ is the quadratic character attached to $F/\Q$.\end{lm}
\begin{proof}
  This is \cite[Proposition 2.5, Proposition 2.6]{yo80}. 
\end{proof}
At the archimedean place $\infty$, we have identified $H(\R)$ with 
$\bbH^\x\x\bbH/\R^\x$ and $V_\infty$ with $\bbH$ via the isomorphisms fixed in \eqref{E:splittrivialization} so that $H(\R)$ acts on $V_\infty=\bbH$ by $\varrho(a,d)x=axd^{-1}$. To define the archimedean test function $\varphi\in\cS(\bfX_\infty)=\cS(\bbH^{\oplus 2})$, we need to introduce several special polynomials.  
Let ${\mathbf p}: {\rm M}_2(\C)^{{\rm Tr}=0}\to \C[X_1,Y_1]_2$ be the map defined by
\[     {\mathbf p}\left( \begin{pmatrix} a & b \\ c & -a \end{pmatrix}\right)
  =  - bX_1^2 + 2aX_1Y_1+cY_1^2.
\]It is easy to see that \beq\label{E:equiv1}{\mathbf p}(gxg^{-1})=\tau_1(g)({\mathbf p}(x))\text{ for } g\in\GL_2(\C).\eeq
Define ${\mathbf q}:{\rm M}_2(\C)\to \C[X_1,Y_1]_1\otimes \C[X_2,Y_2]_1$ by
\begin{align*}
   {\mathbf q}(x) = {\rm Tr}\left(  x^* \begin{pmatrix} 0 & 1 \\ -1 & 0 \end{pmatrix} W\right)  
   \quad (W= \begin{pmatrix}  X_1\otimes X_2 & X_1\otimes Y_2 \\ Y_1\otimes X_2 & Y_1\otimes Y_2  \end{pmatrix}).  
\end{align*}
In particular, \begin{align*}
       {\mathbf q}\left( \begin{pmatrix} z & w \\ -\bar{w} & \bar{z} \end{pmatrix} \right) 
   =  \bar{z} Y_1\otimes X_2 + w X_1\otimes X_2 - z X_1\otimes Y_2 + \bar{w} Y_1\otimes Y_2. 
\end{align*}
For each integer $\alpha$ with $0\leq \alpha \leq 2k_2$, 
define $P^\alpha_{\underline{k}}: {\rm M}_2(\C)^{\oplus 2} \to \C[X_1,Y_1]_{2k_1}\otimes \C[X_2,Y_2]_{2k_2}$ by
\begin{align*}
  P^\alpha_{\underline{k}}(x_1,x_2) = {\mathbf p}(x_1x^*_2-\frac{1}{2}{\rm Tr}(x_1x^*_2)\cdot\bfone_2)^{k_1-k_2} \cdot {\mathbf q}(x_1)^\alpha {\mathbf q}(x_2)^{2k_2-\alpha}.  
\end{align*}
Define $P_{\underline{k}}:{\rm M}_2(\C)^{\oplus 2}\to \C[X_1,Y_1]_{2k_1}\otimes \C[X_2,Y_2]_{2k_2}\otimes_\C\C[X,Y]_{2k_2}$ to be the map 
\begin{align*}
   P_{\underline{k}}(x_1,x_2) =& \sum^{2k_2}_{\alpha =0} P^\alpha_k(x_1,x_2)\otimes \binom{2k_2}{\alpha} X^\alpha Y^{2k_2-\alpha}.
\end{align*}
The archimedean Bruhat-Schwartz function $\varphi_\infty: {\mathbf X}_\infty={\mathbb H}^{\oplus 2}\to \C[X_1,Y_1]_{2k_1}\otimes \C[X_2,Y_2]_{2k_2}\otimes_\C\C[X,Y]_{2k_2}$ is defined by
\begin{align}\label{infinitetest}
  \varphi_\infty(x) = e^{-2\pi( {\rm n}(x_1) + {\rm n}(x_2))} \cdot P_{\underline{k}}(x_1,x_2) . 
\end{align}
To be explicit, the polynomials $P_{\underline{k}}: {\mathbb H}^{\oplus 2}  \to \C[X_1,Y_1]_{2k_1}\otimes \C[X_2,Y_2]_{2k_2}\otimes_\C\C[X,Y]_{2k_2}$ can be written down in the following form:
\begin{align*}
     &  P_{\underline{k}} ( \begin{pmatrix} z_1 & w_1 \\ -\bar{w}_1 & \bar{z}_1 \end{pmatrix} , \begin{pmatrix} z_2 & w_2 \\ -\bar{w}_2 & \bar{z}_2 \end{pmatrix}  )  \\
  = & ( ( z_1\bar{z}_2+w_1\bar{w}_2-\bar{w}_1w_2-\bar{z}_1z_2 ) X_1Y_1 + (z_1w_2-w_1z_2 ) X^2_1 + (\bar{z}_1\bar{w}_2-\bar{z}_2\bar{w}_1 )Y^2_1 )^{k_1-k_2}    \\
     & \times \sum^{2k_2}_{\alpha=0} (\bar{z}_1Y_1\otimes X_2 + w_1 X_1\otimes X_2 -  z_1 X_1\otimes Y_2 + \bar{w}_1 Y_1\otimes Y_2)^\alpha  \\
     & \quad \quad \times (\bar{z}_2Y_1\otimes X_2 + w_2 X_1\otimes X_2 -  z_2 X_1\otimes Y_2 + \bar{w}_2 Y_1\otimes Y_2)^{2k_2-\alpha}\binom{2k_2}{\alpha}X^{\alpha}Y^{2k_2-\alpha}.  
\end{align*}
Note that the coefficients of $P_{\ul{k}}$ are integral polynomials in $\stt{z_i,\ol{z}_i, w_i,\ol{w}_i}_{i=1,2}$.
\begin{lm}\label{lem1Y}
For $(h,g)\in H^{(1)}(\C)\times {\rm GL}_2(\C)$, we have
\begin{align*}
  P_{\underline{k}}(\varrho(h)(x_1,x_2)g) = \tau_{\ul{k}}(h)\otimes\rho_{(k_1+k_2,k_1-k_2)}({}^{\rm t}g)(P_{\underline{k}}(x_1,x_2)).
\end{align*}
\end{lm}
\begin{proof}Note that $H^{(1)}(\C)=\stt{(a,d)\in\GL_2(\C)^{\oplus 2}\mid \det a=\det d}$. The assertion for $h\in H^{(1)}(\C)$ can be verified by \eqref{E:equiv1}, and the assertion for $\GL_2(\C)$ can be checked by a direct computation.
\end{proof}

\begin{lm}\label{lem2Y}
The map $P_{\underline{k}}:{\mathbb H}^{\oplus 2}\to \C[X_1,Y_1]_{2k_1}\otimes \C[X_2,Y_2]_{2k_2}\otimes_\C\C[X,Y]_{2k_2}$ is a vector-valued pluri-harmonic polynomial.
\end{lm}
\begin{proof}We recall the definition of pluri-harmonic polynomials given in \cite[p.\,18]{kv78}.
Let $\Delta_{11}, \Delta_{22}$ and $\Delta_{12}$ be the differential operators on $\C[z_1, \bar{z}_1, w_1, \bar{w}_1, z_2, \bar{z}_2, w_2, \bar{w}_2]$ defined by
\begin{align*}
   \Delta_{ii} = \frac{\partial^2}{\partial z_i \partial\bar{z}_i} + \frac{\partial^2}{\partial w_i \partial\bar{w}_i} \quad (i=1,2), 
   \quad 
   \Delta_{12} =   \frac{\partial^2}{\partial z_1\partial\bar{z}_2} + \frac{\partial^2}{\partial \bar{z}_1\partial z_2}
                     + \frac{\partial^2}{\partial w_1\partial\bar{w}_2} + \frac{\partial^2}{\partial \bar{w}_1\partial w_2}.  
\end{align*}
Then a polynomial $P\in\C[z_1, \bar{z}_1, w_1, \bar{w}_1, z_2, \bar{z}_2, w_2, \bar{w}_2]$ is said to be pluri-harmonic if and only if \[\Delta_{ij}P=0\text{ for all }i,j\in \{1,2\}.\]
Now the lemma follows from a direct and elementary computation of $\Delta_{ij}P_{\underline{k}}$. We leave it to the readers.
\end{proof}

\begin{lem}\label{lem3Y}
{\itshape
Let $P(x)$ be a pluri-harmonic polynomial on ${\mathbf X}_\infty={\mathbb H}^{\oplus 2}$ and let
\begin{align*}
\varphi(x) = P(x) \cdot e^{-2\pi{\rm Tr}(S_x)}. 
\end{align*}
For $u=A+\sqrt{-1} B\in {\rm U}_2(\R)\subset {\rm GL}_2(\C)$, we have 
\begin{align*}
  \omega_{V_\infty}( \begin{pmatrix} A & B \\ -B & A \end{pmatrix}) \varphi(x) = \varphi(xu)\cdot (\det u)^2. 
\end{align*}
}
\end{lem}
\begin{proof}
By \cite[Lemma 4.5]{kv78}, we have
\begin{align*}
  \int_{{\mathbf X}_\infty} \psi(x{}^{\rm t}y)\psi(\frac{1}{2}xz{}^{\rm t}x) P(x) dx 
  = {\rm det} (\frac{z}{\sqrt{-1}})^{-2} \cdot \psi(-\frac{1}{2}yz^{-1}{}^{\rm t}y) P(-yz^{-1}). 
\end{align*}
We thus obtain
\begin{align*}
     \omega_{V_\infty}( \begin{pmatrix} & 1 \\ -1 &  \end{pmatrix} u(b)) \varphi
  = P(-x(b+\sqrt{-1})^{-1}) \psi(-\frac{1}{2}\langle x,x(b+\sqrt{-1})^{-1} \rangle) \cdot {\rm det}(\frac{b+\sqrt{-1}}{\sqrt{-1}})^2.   
\end{align*}
Using the decomposition
\begin{align*}
   \begin{pmatrix} A & B \\ -B & A \end{pmatrix} 
 =\begin{pmatrix} 1 & -AB^{-1} \\ 0 & 1 \end{pmatrix}
   \begin{pmatrix} {}^{\rm t}B^{-1} & 0 \\ 0 & B \end{pmatrix}
   \begin{pmatrix} & 1 \\ -1 &  \end{pmatrix}
   \begin{pmatrix} 1 & -B^{-1}A \\ 0 & 1 \end{pmatrix},
\end{align*}
one shows the lemma by a straightforward calculation.
\end{proof}

\begin{lem}\label{lem4Y}
{\itshape
For $(h,\bfk)\in H^{(1)}(\R)\times \bfK_\infty$,
\begin{align*}
   (\omega_\infty(h,\bfk)\varphi_\infty)(x)= &\tau_{\ul{k}}(h^{-1})\otimes \rho_\kappa({}^{\rm t}\bfk)(\varphi_\infty (x)).  
\end{align*}
}
\end{lem}
\begin{proof}Recall that $\kappa = (k_1+k_2+2, k_1-k_2+2)$. It follows immediately from Lemma \ref{lem1Y}, Lemma \ref{lem2Y} and Lemma \ref{lem3Y}.
\end{proof}

\subsection{The Fourier expansion of Yoshida lifts}\label{SS:FCYoshida}
With the above distinguished test function 
  $\varphi:=\varphi_\infty\otimes \varphi_f\in {\mathcal S}({\mathbf X}_{\mathbf A})\otimes {\mathcal W}_{\underline{k}}(\C)\otimes {\mathcal L}_\kappa(\C)$
  defined in (\ref{finitetest}) and (\ref{infinitetest}), we see that the Yoshida lift $\theta_{\mathbf f}: {\rm GSp}_4({\mathbf A})\to {\mathcal L}_\kappa(\C)$ attached to ${\mathbf f}$ is defined by
\begin{align*}
   \theta_{\mathbf f} = \theta(-;{\mathbf f},\varphi)\in \cA_\kappa(\GSp_4(\A),N_F,\chi_{F/\Q}).  
\end{align*}
is a (adelic) Siegel modular form of weight $\kappa$, level $N_F$ and type $\chi_{F/\Q}$ in view of \lmref{Yoshidalev} and \lmref{lem4Y}. Define the \emph{classical Yoshida lift} $\theta^\ast_{\mathbf f}: {\mathfrak H}_2  \to {\mathcal L}_\kappa(\C)$ by 
\begin{align*}
      \theta^\ast_{\mathbf f}(Z) 
  =  \rho_\kappa( J(g_\infty, {\mathbf i}) )\theta_{\mathbf f}(g_\infty) \quad (g_\infty\in {\rm Sp}_4(\R), g_\infty\cdot{\mathbf i} = Z). 
\end{align*}
Let $\Gamma^{(2)}_0(N_F):=\Sp_4(\Q)\cap U^{(2)}_0(N_F)\subset \Sp_4(\Z)$. By definition, 
\[\theta^*_\bff(\gamma\cdot Z)=\rho_\kappa(J(\gamma,Z))\theta^*_\bff(Z)\]
for $\gamma\in\Gamma^{(2)}_0(N_F)$. 

We recall the calculation of Fourier coefficients of $\theta^*_\bff(Z)$ following \cite[\S\,3]{yo84}. Let $\xi\in\GL_2(\A)$ and $\nu=\nu(t)\in\A_f^\x$ for some $t\in H(\A_f)$. Put $g=\pDII{\xi}{\nu{}^{\rm t}\xi^{-1}}\in\GSp_4(\A)$. We have \begin{align*}
       \bfW_{\theta_{\mathbf f}, S}(g)
= & \int_{[U]}  \theta_{\mathbf f}(ug) \psi_S(u) {\rm d}u\\
= & \int_{[U]} \ol{\psi_{S_x}(u)}\psi_S(u){\rm d}u \int_{[H^{(1)}]}\sum_{x\in {\mathbf X}} \langle \omega(ht, g)\varphi(x), {\mathbf f}(ht) \rangle_{2\ul{k}}{\rm d}h  \\
=&\int_{[H^{(1)}]}\sum_{x\in \bfX,\,S=S_{\bfz}} \langle \omega(ht,g)\varphi(\bfx), {\mathbf f}(ht) \rangle_{2\ul{k}}{\rm d}h.
\end{align*}
Therefore, if $\bfW_{\theta_\bff,S}(g)\not=0$, then $S=S_{\bfz}$ for some $z\in \bfX$, and $S$ is semi-positive definite. Now let $S=S_{\bfz}$ and put
  \[ H_{\bfz}=\stt{ h\in H^{(1)}\mid \varrho(h)\bfz= \bfz}. \]
It follows from Witt's theorem that
\[    \varrho(H^{(1)}(\Q))\bfz=\stt{x\in \bfX\mid S_x=S_\bfz};\]
we thus obtain 
\beq\label{E:FC1}\begin{aligned}
  \bfW_{\theta_{\mathbf f}, S}(g)=&\int_{[H^{(1)}]}\sum_{\gamma\in H_{\bfz}(\Q)\bksl H^{(1)}(\Q)}\pair{\om(t,g)\varphi(\varrho(h^{-1}\gamma^{-1})z)}{\bff(ht)}_{2\ul{k}}{\rm d}h \\
  =&\int_{H_{\bfz}(\Q)\bksl H^{(1)}(\A)}\pair{\om(t,g)\varphi(\varrho(h^{-1})z)}{\bff(ht)}_{2\ul{k}}\rmd h\\
    =&\chi_{F/\Q}(\det \xi)\abs{\nu^{-1}\det \xi}_\A^2\int_{H_{\bfz}(\Q)\bksl H^{(1)}(\A)}\pair{\varphi(\varrho(t^{-1}h^{-1})z\xi)}{\bff(ht)}_{2\ul{k}}\rmd h.\end{aligned}\eeq
To proceed the computation, we introduce some notation. Define the subset $\Lam_2\subset\cH_2(\Q)$ by 
 \[\Lam_2=\stt{S=\pMX{a}{b/2}{b/2}{c}\mid S\text{ is semi-positive definite with }a,b,c\in\Z}.\]
Define the set $\cE_{\bfz}$ by
 \[\cE_{\bfz}:=\stt{h_f\in H^{(1)}(\A_f)\mid \varrho(h_f^{-1})\bfz\in V(\wh\Z)\oplus V(\wh\Z)}.\]
 Then we have $H_{\bfz}(\A_f)\cE_{\bfz}\cU=\cE_{\bfz}$, and according to \cite[Proposition 1.5]{yo84}, the cardinality $\sharp(H_{\bfz}(\A_f)\bksl \cE_{\bfz}/\cU)$ is finite, so 
 $[\cE_{\bfz}]:=H_{\bfz}(\Q)\bksl \cE_{\bfz}/\cU$ is also a finite set. 
 
 \begin{prop}\label{P:FCYoshida}The classical Yoshida lift $\theta_\bff^*$ has the Fourier expansion 
\[     \theta^\ast_{\mathbf f}(Z) = \sum_{S} {\mathbf a}(S)q^S \quad (q^S={\rm exp}(2\pi\sqrt{-1}{\rm Tr}(SZ))),
\]where $S$ runs over elements in $\Lam_2$ such that $S=S_{\bfz}$ for some $\bfz\in\bfX$ and 
\beq\label{E:classicalFC}\begin{aligned}\bfa(S)=&\sum_{h_f\in [\cE_{\bfz}]}w_{\bfz,h_f}\cdot \pair{P_{\ul{k}}(\bfz)}{\bff(h_f)}_{2\ul{k}},\\
&\quad (w_{\bfz,h_f}:=\sharp(H_{\bfz}(\Q)\cap h_f\cU h_f^{-1})^{-1}).\end{aligned}\eeq
In particular, $\theta^\ast_{\mathbf f}$ is a holomorphic vector-valued Siegel modular form of weight $\Sym^{2k_2}(\C^{\oplus 2})\ot\det^{k_1-k_2+2}$ and level $\Gamma^{(2)}_0(N_F)$. 
 \end{prop}
 \begin{proof}
 Let $Z=X+\sqrt{-1}Y\in\frakH_2$ and choose $\xi_\infty\in\GL_2(\R)$ such that $Y=\xi_\infty{}^{\rm t}\xi_\infty$. Put $\al(\xi_\infty)=\pDII{\xi_\infty}{\rt\xi_\infty^{-1}}$. By \eqref{E:FC},
\begin{align*}     \theta^\ast_{\mathbf f}(Z) =&\sum_{S}\rho_\kappa(J(g_\infty,\mathbf i))\bfW_{\theta_\bff,S}(g_\infty)\quad(g_\infty=u(X)\al(\xi))\\
=&\sum_{S}\rho_\kappa(\rt\xi_\infty^{-1})\bfW_{\theta_\bff,S}(\al(\xi_\infty))\cdot e^{2\pi\sqrt{-1}\Tr(SX)}.
\end{align*}
Suppose that $\bfW_{\theta_\bff,S}(\al(\xi_\infty))\not=0$. Then $S=S_{\bfz}$ for some $\bfz\in\bfX$. Combined with the fact that $\theta_\bff$ is left invariant by unipotent elements in $U^{(2)}_0(N_F)$, we see that $S\in\Lam_2$. Note that by \lmref{lem1Y}, 
\[(\det \xi_\infty)^2\cdot \varphi(z\xi_\infty)=\left(\rho_\kappa({}^{\rm t}\xi_\infty)P_{\ul{k}}(\bfz)\right)\cdot e^{-2\pi\Tr(S_{\bfz} Y)},\]
 where $\kappa=(k_1+k_2+2,k_1-k_2+2).$ Applying \eqref{E:FC1} and \lmref{lem4Y}, we obtain
 \begin{align*}\bfW_{\theta_{\mathbf f}, S_{\bfz}}(\al(\xi_\infty))=&(\det \xi_\infty)^2\int_{H_{\bfz}(\Q)\bksl H^{(1)}(\A_f)}\varphi_f(\varrho(h_f^{-1})\bfz)\pair{\varphi_\infty(z\xi_\infty)}{\bff(h_f)}_{2\ul{k}}{\rm d}h_f\\
 =&\int_{H_{\bfz}(\Q)\bksl H^{(1)}(\A_f)}\bbI_{\cE_{\bfz}}(h_f)\cdot\rho_\kappa({}^{\rm t}\xi_\infty)\pair{P_{\ul{k}}(\bfz)}{\bff(h_f)}_{2\ul{k}}\cdot e^{-2\pi\Tr(S_{\bfz} Y)}\rmd h_f\\
 =&\rho_\kappa({}^{\rm t}\xi_\infty)\left( \sum_{[h_f]\in [\cE_{\bfz}]}w_{\bfz,h_f}\cdot \pair{P_{\ul{k}}(\bfz)}{\bff(h_f)}_{2\ul{k}}\right)\cdot e^{-2\pi\Tr(S_{\bfz} Y)}.
\end{align*}
The proposition follows immediately.
\end{proof}
\begin{Remark}Suppose that $\bff=\bff^\circ$ is the newform associated with an newform $f^{\rm new}$ on $\PGL_2(F_\A)$. Let $\pi$ be the automorphic representation of $\GL_2(F_\A)$ generated by $f^{\rm new}$. When the Galois conjugate $\ol{\pi}$ is not isomorphic to $\pi$, or equivalently $\bff(h)$ is not a scalar of $\bff^\vee(h):=\bff(\ol{h})$, it is well known that  $\theta_{\bff}^*$ is a cusp form, \ie $\bfa(S)=0$ if $\det S=0$ (\cf \cite[Theorem 5.4]{yo80}, \cite[Theorem 1.2]{bsp97}, \cite[Theorem 8.6]{brooks01}).\end{Remark}

\section{Bessel periods of Yoshida lifts}\label{S:Bessel}
\subsection{Bessel periods}\label{BesselNota}
In this section, we let $\bff\in\cA_{\ul{k}}(D_\A^\x,\wh R^\x)$ and calculate the Bessel periods of the Yoshida lift $\theta_{\bff}$ associated to some special imaginary quadratic fields. Let $M$ be an imaginary quadratic field such that $(\Delta_M,N)=1$ and \beqcd{H'}
 \text{ Each  prime factor of $N^-$ is inert in $M$.}
\eeqcd The above assumption assures that the existence of an optimal embedding $\iota:M\hookto D_0$ in the sense that $\iota^{-1}(\cO_{D_0})=\cO_M$. We shall fix an optimal embedding. Let $M=\Q(\sqrt{-d_M})$ and $F=\Q(\sqrt{d_F})$ with $d_M$ and $d_F$ square-free positive integers. Let $K=\Q(\sqrt{-d_Md_F})$. Then we have a natural map $\iota:K\to V\subset D=D_0\ot_\Q F$ such that 
\[\iota(\sqrt{-d_Md_F})=\iota(\sqrt{-d_M})\ot\sqrt{d_F}.\]
Let $d_K$ be the square-free positive integer such that $K=\Q(\sqrt{-d_K})$ and let $\cO_K=\Z\oplus\Z\delta$ with the claasical choice of $\delta$
\beq\label{E:delta}\delta=\begin{cases}\sqrt{-d_K}&\text{ if }-d_K\not\con 1\pmod{4},\\
\frac{1+\sqrt{-d_K}}{2}&\text{ if }-d_K\con 1\pmod{4}.\end{cases}
\eeq
Thus $\delta-\ol{\delta}$ generates the different of $K/\Q$ and $\Im \delta=\sqrt{\Delta_K}/2$. Put
\begin{align*}
 \bfz &= (1, \iota(\delta) )\in {\mathbf X}; \\
 S:&= S_\bfz =\pMX{1}{\frac{\delta+\ol{\delta}}{2}}{\frac{\delta+\ol{\delta}}{2}}{\delta\ol{\delta}}.\end{align*}
We introduce the definition of $S$-th Bessel period according to \cite{fu93}.
Define a $\Q$-algebraic group $T_S$ by
\[
 T_S = \{ g\in {\rm GL}_2  | {}^{\rm t}gSg = \det g S \}.  
\]
Define $\Psi: K^\times \to {\rm GL}_2$ by 
 \begin{align*}
    t\bfz=(t, t\delta) = (1, \delta) \Psi(t)=\bfz \Psi(t) \quad (t\in K^\times),
 \end{align*}
Then $\Psi(K^\x)=T_S$. Let $E:=FK=F(\sqrt{-d_M})$ be a subalgebra of $D$ (via $\iota$) and let \[E_0:=\stt{a\in E\mid \rmN_{E/K}(a)=a\ol{a}^*\in \Q}.\]
Define a morphism $j:E^\x\to \GSp_4,\,t\mapsto j(t)$ by 
\[j(t):=\pDII{\Psi(\rmN_{E/K}(t))}{\Psi(\rmN_{E/K}(\ol{t}))}.\]
Let $\rmd t=\rmd t_\infty\rmd t_f$ be the Haar measure on $E^\x_\A/F^\x_\A$ with $\vol(E^\x_\infty/F^\x_\infty,\rmd t_\infty)=\vol(\wh\cO_E^\x,\rmd t_f)=1$. Let $\rmd a_\infty$ and $\rmd a_f$ be the Haar measures of $H_\bfz(\R)$ and $H_\bfz(\A_f)$ such that 
  $\vol(H_\bfz(\R),\rmd a_\infty)=\vol(H_\bfz(\A_f)\cap \cU,\rmd a_f)=1$ and let $\rmd a:=\rmd a_\infty \rmd a_f$ be the Haar measure on $H_\bfz(\A)$, which will be identified with $(E_0\ot\A)^\x/F_\A^\x$ by the lemma below. 
  \begin{lm}\label{stabz}
We have an isomorphism
\[E_0^\x/F^\x\iso H_\bfz,\quad a\mapsto (a,\rmN_{E/K}(a))\]
 as algebraic groups over $\Q$.
\end{lm}
\begin{proof}
It suffices to show this map is surjective. Let $L$ be a field extension of $\Q$ and let $(a,\alpha)\in H_{\bfz}(L)\subset  (D\ot_\Q L)^\x\times_{(F\ot L)^\x} L^\x$. By definition, 
\[\al^{-1}a\ol{a}^*=1;\quad \al^{-1}a\delta \ol{a}^*=a.\]
This implies that $a$ lies in the centralizer of $E\ot_\Q L$. Since $E\ot_\Q L$ is a maximal commutative subalgebra of $D\ot_\Q L$, we see that $a\in E\ot_\Q L$, and hence $\al=a\ol{a}^*=\rmN_{E/K}(a)\in L$. 
\end{proof}

  For a Siegel modular form $\cF$ of weight $\kappa$, the $S$-th Fourier coefficient $\bfW_{\cF,S}:\GSp_4(\A)\to\cL_\kappa(\C)$ is left invariant by $Z_H(\A)T_S(\Q)$.
For each character $\phi: K^\x\A^\x \bksl K_\A^\x\to \C^\times$, we can define the vector-valued Bessel period ${\mathbf B}_{\cF,S,\phi}: \GSp_4({\mathbf A}) \to {\mathcal L}_\kappa(\C)$ by 
\beq\label{E:dfnBessel}  {\mathbf B}_{\cF,S,\phi}(g) 
   = \int_{[E^\x/E_0^\x]} 
          {\bfW}_{\cF,S}(j(t)g)\phi(\rmN_{E/K}(t)){\rm d}\bar{t},
\eeq
where ${\rm d}\bar{t}$ is the quotient measure ${\rm d}t/{\rm d}a$.
\subsection{Preliminary computation of Bessel periods}\label{S:PBessel}
Let $C$ be a positive integer such that \beqcd{hC}\begin{aligned}
&(C,N\Delta_K)=1;\\
 &\text{Every prime factor $p$ of $C$ is split in either $F$ or $M$},  
 \end{aligned}\eeqcd
and put  \beq\label{E:dfnxiC}
 \xi_C=\pDII{1}{C}\in\GL_2(\A_f);\quad g_C=\pDII{\xi_C}{\rt\xi_C^{-1}}\in\Sp_4(\A_f).\eeq
Define the subset $\cE_{\bfz,C}\subset H^{(1)}(\A_f)$ to be $\cE_{\bfz,C}=\prod_{p<\infty}\cE_{\bfz,C,p}$, where  \[
   \cE_{\bfz,C,p} =  \left\{ h\in H^{(1)}(\Q_p) \mid \varrho(h^{-1})\bfz \xi_{C,p} \in V(\Zp)\oplus V(\Zp) \right\}.
\]
It is clear that $H_\bfz(\A_f)\cE_{\bfz,C}\cU=\cE_{\bfz,C}$, and by definition, $\varphi_f(\varrho(h^{-1})\bfz\xi_{C,f})=\bbI_{\cE_{\bfz,C}}(h)$ for $h\in H^{(1)}(\A_f)$.
 
\begin{prop}\label{P:Bessel1} We have
\[\bfB_{\theta_\bff, S,\phi} (g_C) 
 =C^{-2}\int_{  H_\bfz(\A_f) \backslash  \cE_{\bfz,C}} \int_{[E^\x/F^\x]}   
        \langle \varphi_\infty(\bfz  ), {\mathbf f}(t h_f) \rangle_{2\ul{k}}\cdot
        \phi({\rm N}_{E/K}(t)){\rm d}t{\rm d}\ol{h}_f.\]
 Here $\rmd \ol{h}_f=\rmd h_f/\rmd a_f$.
\end{prop}
\begin{proof}
For simplicity, write $\rmN=\rmN_{E/K}$. Since $\nu(j(t))=\rmN_{E/\Q}(t)=\nu(\varrho(t))$, applying the formula \eqref{E:FC1}, we find that $C^2\cdot \mathbf B_{\theta_\bff,S,\phi} (g_C)$ is equal to 
\begin{align*}
   & C^2\cdot \int_{[E^\x/E_0^\x]} \bfW_{\theta_{\mathbf f}, S}(j(t)g_C) \phi(\rmN(t)) {\rm d}\bar{t}  \\
  = &
        \int_{ H_\bfz(\Q) \backslash H^{(1)}(\A)}  \int_{ [E^\x/E_0^\x]} \langle \varphi(\varrho( t^{-1}h^{-1} \bfz\Psi(\rmN(t))\xi_C ), {\mathbf f}( ht) \rangle_{2\ul{k}}\cdot
        \phi( \rmN(t)) {\rm d}\bar{t}{\rm d}h.
\end{align*}
Using the fact that $\bfz\Psi(\rmN(t))=\varrho(t)\bfz$ and the identification $E_0^\x/F^\x\iso H_\bfz$ in \lmref{stabz}, the above double integral is equal to
 \begin{align*}    &\int_{ H_\bfz(\A) \backslash H^{(1)}(\A)}\int_{[E_0^\x/F^\x]} \int_{ [E^\x/E_0^\x]} \langle \varphi( \varrho(t^{-1}a^{-1}h^{-1}ta) \bfz\xi_C), {\mathbf f}(h(ta,\rmN(a))) \rangle_{2\ul{k}}\cdot
        \phi( \rmN(t)) {\rm d}\bar{t}{\rm d}a{\rm d}\ol{h}\\
   =&\int_{ H_{\bfz}(\A) \backslash H^{(1)}(\A)} \int_{ [E^\x/F^\x]}    \langle \varphi(\varrho(t^{-1}h^{-1}t) \bfz \xi_C), {\mathbf f}( ht) \rangle_{2\ul{k}}\cdot
        \phi( \rmN(t))\rmd t{\rm d}\ol{h}.
        \end{align*}
The above equality holds since $E^\x$ is commutative and $\phi$ is trivial on $\A^\x$. Making change of variable $h\mapsto tht^{-1}$ and applying \lmref{lem4Y}, we obtain
\begin{align*}&C^2\cdot \bfB_{\theta_\bff, S,\phi} (g_C) \\
=& \int_{ H_{\bfz}(\A) \backslash H^{(1)}(\A)} \int_{ [E^\x/F^\x]}    \langle \varphi(\varrho(h^{-1}) \bfz \xi_C), {\mathbf f}( th) \rangle_{2\ul{k}}\cdot
        \phi( \rmN(t))\rmd t{\rm d}\ol{h}\\
=&\int_{  H_\bfz(\A_f) \backslash  H^{(1)}(\A_f)}\int_{ [E^\x/F^\x] }\varphi_f(\varrho(h_f^{-1})\bfz\xi_{C})  
       \langle \varphi_\infty(\bfz  ), {\mathbf f}(t h_f) \rangle_{2\ul{k}}\cdot
        \phi({\rm N}(t)){\rm d}t{\rm d}\ol{h}_f.    
         \end{align*}
This completes the proof.
\end{proof}

\subsection{Determination of $\cE_{\bfz,C}$}\label{seccoset}  
Let $p$ be a rational prime and let $\cU_p$ be the $p$-component of the open-compact subgroup $\cU$. In this subsection, we give the explicit description of the double cosets $[\cE_{\bfz,C,p}]:=H_\bfz(\Q_p)\backslash \cE_{\bfz,C,p} / {\mathcal U}_p$, which will be needed for the further computation of Bessel periods. In addition to \eqref{H'}, we assume that $M$ satisfies the following condition:
\beqcd{rFK}\text{ If $p$ is ramified in $F$ and $K$, then $p$ is inert in $M$.}\eeqcd 
If $(\Delta_F,\Delta_K)=1$, then \eqref{rFK} holds automatically. In what follows, for $p\ndivides N^-$, we identify $D_{0,p}$ with ${\rm M}_2(\Qp)$ via the fixed isomorphism $\Phi_p$ in \subsecref{SS:DATA}. Put
\[\delta_F=\sqrt{d_F};\quad \delta_M=\iota(\sqrt{-d_M}).\]

  \begin{lem}\label{varsigma}For $p\ndivides N^-$, there exists $\varsigma_p\in \GL_2(\cO_{F_p})$ satisfying the following condition:
\begin{enumerate}
\item \label{varsigma(i)}
If $p$ is split in $M$, then $\varsigma_p\in \GL_2(\Zp)$ and
\[              \varsigma^{-1}_p\iota(\delta_M) \varsigma_p  
           = \begin{pmatrix} \delta_M& \\ & -\delta_M \end{pmatrix}. \]                                                
\item \label{varsigma(ii)}
If $p$ is non-split in $M$ and in $K$, then $\varsigma_p \in \GL_2(\Zp)$ such that
    \[
           \varsigma_p^{-1}\iota(\delta_M)\varsigma_p=\begin{cases}
            \pMX{1}{\frac{-1+d_M}{2}}{2}{-1} &\text{ if $p$ is inert in $M$},\\
            \pMX{0}{-d_M}{1}{0}&\text{ if $p$ is ramified in $M$}.\end{cases}
               \]
             \item If $p$ is inert in $M$ and split in $K$, then $\det\varsigma_p\in\Zp^\x$ and 
             \[\varsigma_p^{-1}\ol{\varsigma_p}=\pMX{0}{\delta_F}{-\delta_F^{-1}}{0}\quad\varsigma_p^{-1}\iota(\delta_M)\varsigma_p=\pDII{\delta_M}{-\delta_M}.\]
\end{enumerate}
\end{lem}
\begin{proof}(i) and (ii) are standard facts. To see (iii), note that there exists $g\in\SL_2(\cO_{F_p})$ such that $g^{-1}\iota(\delta_K)g=\pDII{\delta_K}{-\delta_K}\in\GL_2(\Qp)$ as $p$ is split in $K$. Let $b=g^{-1}\ol{g}\in\SL_2(\cO_{F_p})$. We have $\ol{b}=b^{-1}$ and 
\[(-1)\ol{g}^{-1}\iota(\delta_K)\ol{g}=\pDII{\delta_K}{-\delta_K}=g^{-1}\iota(\delta_K)g.\]
It follows that $b\pDII{-1}{1}=\pDII{1}{-1}b$, and hence $b=\pMX{0}{a}{-a^{-1}}{0}$ for some $a\in \cO_{F_p}^\x$ with $\ol{a}=-a$. Replacing $g$ by $g\pDII{1}{\delta_Fa^{-1}}$ for some $a\in F_p^\x$, we find that the conjugation by $g$ sends $E_p$ into diagnoal matrices,  
\[g^{-1}\ol{g}=\pMX{0}{\delta_F}{-\delta_F^{-1}}{0};\,\det g=\delta_F a^{-1}\in \Zp^\x.\]
Then this $g$ satisfies the conditions in (iii).
\end{proof}
\begin{defn}\label{D:cosets}Let $c_p={\rm ord}_pC$. For each prime $p$ such that either $p$ is prime to $N^+$ or $p$ is split in $M$, we define the subset $\widetilde{\cE}_{\bfz,C,p}\subset H^{(1)}(\Qp)$ as follows:
\begin{enumerate}
\item  If $p\ndivides N^+N^-$ is split in $M$, then 
\[             \widetilde{\cE}_{\bfz,C, p} = 
             \left\{  (\varsigma_p \begin{pmatrix} 1 & p^{-j} \\ 0 & 1 \end{pmatrix},\det\varsigma_p)  \mid  0\leq j \leq c_p      \right\}.    \]
\item If $p\mid N^+$ is split in $M$, then  \begin{align*}
             \widetilde{\cE}_{\bfz,C, p} = \left\{ (\varsigma_p,\det\varsigma_p), (\varsigma_p \pMX{0}{1}{-1}{0},\det\varsigma_p)\right\}.
         \end{align*}
\item 
If $p\ndivides N^+N^-$ is non-split in $M$, then 
\[             \widetilde{\cE}_{\bfz,C, p} = \begin{cases} \left\{  (\varsigma_p \pDII{p^{-j}}{1},p^{-j}\det\varsigma_p ) \mid 0\leq j \leq c_p     \right\}   & \text{ if $p$ is split in $F$, }  \\
                                                         \left\{  (\varsigma_p,\det\varsigma_p)   \right\}    & \text{ if $p$ is non-split in $F$.}
                                                           \end{cases}          \]
\item If $p\mid N^-$, let $\pi_{D_p}$ be an element of $D_{0,p}$ with $\pi_{D_p}\pi_{D_p}^*=p$, and put
         \begin{align*}
             \widetilde{\cE}_{\bfz,C, p} =   \{ (\bfone_2,1), (\pi_{D_p},p) \}.
         \end{align*}
\end{enumerate}
\end{defn}
We record here the following integral analogue of Skolem-Noether theorem.
\begin{lm}\label{L:SkolemN}Let $F/\Qp$ be a finite extension and $E$ be a quadratic field over $F$. Let $\cO$ be an order of $\cO_E$ and $R={\rm M}_2(\cO_F)$. If $f,f':\cO\hookto R$ are two optimal embeddings of $\cO$ into $R$, then $f'(x)=u^{-1}f(x)u$ for some $u\in R^\x$.
\end{lm}
\begin{proof}This is a special case of \cite[Corollary 2.6 (i)]{hijikata74} (See also the last paragraph of \cite[p.1158]{gross88}). \end{proof}

\begin{prop}\label{coset.Y}
 If $p\ndivides N^+$ or $p$ is split in $M$, then the set $\wtd\cE_{\bfz,C,p}$ is a complete set of representatives of $[\cE_{\bfz,C,p}]$. If $p\mid N^+$ is non-split in $M$, then $\cE_{\bfz,C,p}$ is the empty set.
\end{prop}
\begin{proof}  Let $\varsigma=(\varsigma_p,\det\varsigma_p)\in H(\Qp)$ and \[\bfz^\prime =\varrho(\varsigma^{-1})\bfz,\quad H_{\bfz^\prime} :=\varsigma^{-1}H_{\bfz}(\Qp)\varsigma;\quad \cE_{\bfz',C,p}=\varsigma^{-1}\cE_{\bfz,C,p}.\] 
Suppose that $\cE_{\bfz',C,p}$ is not empty and let $h\in \cE_{\bfz',C,p}$, or equivalently 
\beq\label{E:E1} \varrho(h^{-1})\bfz'\in R_p\oplus C^{-1}\cdot R_p.\eeq
Denote by $[h]$ the double coset $H_{\bfz'}h \cU_p$. The task is to show that the class $[h]$ can be represented by some element in $\varsigma^{-1}\wtd\cE_{\bfz,C,p}$ and that $p\ndivides N^+$ if $p$ is non-split in $M$.

 {\bf Case (i) $p$ is split in $M$:} In this case, $p$ is unramified in $F$ and $K$ by \eqref{rFK}, and one verifies that $\bfz'=(1,\pDII{\delta}{\ol{\delta}})$ and 
    \begin{align*}
      H_{\bfz^\prime} = \left\{ (\begin{pmatrix} a & \\ & d  \end{pmatrix}, \alpha ) \in {\rm GL}_2(F_p) \times_{F^\times_p}\Q^\times_p : a\ol{d}=\alpha  \right\}.
    \end{align*}
Using the Iwasawa decomposition, one can verify that $[h]$ can be represented by an element $h_1$ of the form 
\[h_1=(\pMX{1}{-y}{0}{1}\pDII{a^{-1}}{1}u,1)\quad (a\ol{a}=1,\,y\in F_p,\,u\in H^{(1)}(\Zp)),\]
where $H^{(1)}(\Zp)=H^{(1)}(\Qp)\cap (\GL_2(\cO_{F_p})\times_{\cO_{F_p}^\x}\Zp^\x)$. 
Then \eqref{E:E1} implies that 
\begin{align*}\varrho(h_1^{-1})\bfz'=&(\pMX{a}{y-\ol{y}}{0}{\ol{a}},\pMX{a\delta}{y\ol{\delta}-\ol{y}\delta}{0}{\ol{a\delta}})\in {\rm M}_2(\cO_{F_p})\oplus {\rm M}_2(C^{-1}\cO_{F_p}).
\end{align*}
Since $a\ol{a}=1$ and $p$ is unramified in $E$, from the above relation we can deduce that 
\[a\in\cO_{F_p}^\x;\quad  y\con y_1\pmod{\cO_{F_p}}\text{ with }y_1\in C^{-1}\Zp,\,-j\leq \Ord_p(y_1)\leq 0.\]
Writing $y_1=-p^{-j}v\ol{v}$ with $0\leq j\leq c_p$ and $v\in \cO_{F_p}^\x$, it follows that 
\[[h]=[h_1]=[(\pDII{v}{\ol{v}^{-1}}\pMX{1}{p^{-j}}{0}{1}\pDII{v^{-1}}{\ol{v}}u,1)]=[(\pMX{1}{p^{-j}}{0}{1}u_1,1)]\]
for some $u_1\in H^{(1)}(\Zp)$. If $p\ndivides N^+$, then $\cU_p=H^{(1)}(\Zp)$, and hence $[h]$ can be represented by some element in $\varsigma^{-1}\wtd\cE_{\bfz,C,p}$.

{\bf Case (ii) $p\mid N^+$ is split in $M$:} In this case, $c_p=0$, so from the discussion in the previous case we see that the class $[h]$ can be represented by some element in $\GL_2(\cO_{F_p})$. Now we claim that $[h]$ can be represented by some element $h_2$ of the form
\[h_2=(\pMX{1}{0}{x}{1},1) \text{ or }h_2=(\pMX{x}{1}{-1}{0},1),\quad x\in \cO_{F_p}.\]
 To see the claim, recall that 
\[R_p=\stt{\pMX{a}{b}{c}{d}\in {\rm M}_2(\cO_{F_p})\mid c\con 0\pmod{N^+}},\]
 and note that for a finite extension $L/\Qp$ with a uniformizer $\varpi$, we have the coset decomposition
\beq\label{E:Iwasawa}   \GL_2(\cO_L)
 = \disjoint_{x\in \varpi\cO_L}\begin{pmatrix} 1 &0 \\ x& 1  \end{pmatrix}B(\cO_L) 
        \disjoint_{x\in\cO_L}   \begin{pmatrix} x  & 1 \\-1 &   0\end{pmatrix} B(\cO_L),
\eeq where $B(\cO_L)=\stt{\begin{pmatrix} a & b \\ 0& d \end{pmatrix}  |  a, d\in\cO_L^\x, b\in \cO_L }$, so it suffices to consider the case where $F_p=\Qp \oplus \Qp $ and $[h]$ is represented by 
\[h_3=(\pMX{(1,y)}{(0,1)}{(x,-1)}{(1,0)},1),\,x\in p\Zp,y\in\Zp.\]
Then $\varrho(h_3^{-1})\bfz'\xi_{C,p}\in R_p\ot R_p$ implies that
\[\pMX{(y,-x)}{(1,-1)}{(1-xy,xy-1)}{(-x,y)}\in R_p\implies (1-xy,xy-1)\in N^+, \]
which is a contradiction as $1-xy\in\Zp^\x$. This proves the claim. We proceed the argument. An easy computation shows that \begin{align*}
\varrho(h_2^{-1})\bfz'&=\begin{cases}(\pMX{1}{0}{x-\ol{x}}{1},\pMX{\delta}{0}{\ol{x\delta}-x\delta}{\ol{\delta}})&\text{ if }h_2=\pMX{1}{0}{x}{1},\\
(\pMX{1}{0}{x-\ol{x}}{1},\pMX{\ol{\delta}}{0}{x\ol{\delta}-\ol{x}\delta}{\delta})&\text{ if }h_2=\pMX{x}{1}{-1}{0}.\end{cases}
\end{align*} The condition $\varrho(h_2^{-1})\bfz'\in R_p\oplus R_p$ implies that  $x\con 0\pmod{N^+}$, and hence 
\[[\cE_{\bfz',C,p}]=\stt{[(\bfone_2,1)],[\pMX{0}{1}{-1}{0},1]}.\]

{\bf Case (iii) $p\ndivides N^-$ is non-split in $M$:} We also need to show $p\ndivides N^+$ in this case. First consider the subcase where $p=\frakp\ol{\frakp}$ is split in $F$ (so $p$ is non-split in $K$). Recall that we make the identifications $i_\frakp:(\GL_2(\Qp)\x\GL_2(\Qp))/\Qp^\x\iso H(\Qp)$ and $j_\frakp:{\rm M}_2(\Qp)\iso V_p$ via the isomorphisms corresponding to $\frakp$ in \eqref{E:splittrivialization} and that $(g_1,g_2)\in H(\Qp)=(\GL_2(\Qp)\x\GL_2(\Qp))/\Qp^\x$ acts on $V_p={\rm M}_2(\Qp)$ by $\varrho(g_1,g_2)x=g_1xg_2^{-1}$.  Write $\bfz'= (1,\delta')$ and $\delta'\in{\rm M}_2(\Qp)$. Let
\[K_{\bfz'}:=\stt{y\in {\rm M}_2(\Qp)\mid y\delta'=\delta'y}=\Qp(\delta').\] Then $K_{\bfz'}\iso K_p$ is a quadratic field over $\Qp$ as $p$ is non-split in $K$, and by definition
\[H_{\bfz'}=K^\x_{\bfz'}.\]
For $h=(h_1,h_2)\in \cE_{\bfz,C,p}$, $\det h_1=\det h_2$, and we can verify that the class $[h]$ can be represented by $(g,g)$ for some $g\in \GL_2(\Qp)$. For a non-negative integer $m$, let $\cO_{\bfz',m}=\Zp+p^m\cO_{K_{\bfz'}}$ be the order of $K_{\bfz'}$ with conductor $p^{m}$. Let $\gamma:K_{\bfz'}\to {\rm M}_2(\Qp)$ be the conjugation map $\gamma(x)=g^{-1}xg$ and let $p^j$ be the conductor of the order $\gamma^{-1}({\rm M}_2(\Zp))\cap K_{\bfz'}$. Thus, $\gamma$ is an optimal embedding of $\cO_{\bfz',j}$ into ${\rm M}_2(\Zp)$. On the other hand, \eqref{E:E1} implies that $g^{-1}\delta' g\in C^{-1}R_p$, or equivalently 
\[ \gamma(\cO_{\bfz',c_p})=g^{-1}\cO_{\bfz',c_p}g\subset R_p.
\]
This implies that $0\leq j\leq c_p$. If $p\mid N^+$, then $p$ is inert in $K$ and $c_p=j=0$.     We see that $\gamma$ is an optimal embedding of $\cO_{\bfz'}$ into $R_p$ the Eichler order of level $N^+$ in ${\rm M}_2(\Zp)$. This implies that $p$ is split in $K$, which is a contradiction. 

Therefore, we have $p\ndivides N^+$. Then $R_p={\rm M}_2(\Zp)$, and by our choice of $\varsigma_p$, one can verify directly that the conjugation $\gamma':K_{\bfz'}\to {\rm M}_2(\Qp),\,\gamma'(x)=\pDII{p^j}{1}x\pDII{p^{-j}}{1}$ also induces an optimal embedding of $\cO_{j}$ into ${\rm M}_2(\Zp)$. By \lmref{L:SkolemN}, $\gamma(x)=u^{-1}\gamma'(x)u$ for some $u\in \GL_2(\Zp)$. It follows that 
\beq\label{E:caseiii} g\in K_{\bfz'}^\x\pDII{p^{-j}}{1}\GL_2(\Zp)\text{ for some }0\leq j\leq c_p,\eeq
hence $[h]=[(g,g)]$ is represented by 
\[[i_\frakp(\pDII{p^{-j}}{1},\pDII{p^{-j}}{1})]=[(\pDII{p^{-j}}{1},p^{-j})],\,0\leq j\leq c_p.\]

 Now we treat the subcase where $p$ is inert in $F$ but is split in $K$. Then $c_p=0$ by \eqref{hC}. One verifies that \begin{align*} \bfz'=&(\pMX{0}{\delta_F}{-\delta_F^{-1}}{0},\pMX{0}{\delta\delta_F}{-\ol{\delta}\delta_F^{-1}}{0});\\
H_{\bfz'}=&\stt{(\pDII{a}{d},\al)\mid a\ol{a}=d\ol{d}=\al}.\end{align*}
By Iwasawa decomposition, the class $[h]$ can be represented by $h_3\cdot (u,1)$, where \begin{align*}h_3=&(\pMX{p^n}{p^ny}{0}{1},p^n);\quad u\in\SL_2(\cO_{F_p}).
\end{align*}
Put $s=yp^n\delta_F^{-1}$. Since $\varrho(h^{-1})\bfz'\in R_p\oplus R_p$, we have
\begin{align*}  
\varrho(h_3^{-1})\bfz'
=&(\pMX{s}{p^{-n}\delta_F(1-s\ol{s})}{-p^n\delta_F^{-1}}{\ol{s}},
\pMX{s\ol{\delta}}{p^{-n}\delta_F(\delta-s\ol{s}\ol{\delta})}{-p^n\delta_F^{-1}\ol{\delta}}{\ol{s\delta}})\\
&\in {\rm M}_2(\cO_{F_p})\oplus {\rm M}_2(\cO_{F_p}).
\end{align*}
Note that $\delta_F\in\cO_{F_p}^\x$ and $\delta-\ol{\delta}\in\Zp^\x$ as $p$ is unramified in $F$ and $K$. The above implies that  
$s\in\cO_{F_p}$, $s\ol{s}\con 1\pmod{p^n}$ and $\delta\con\ol{\delta}\pmod{p^n}$.
 We conclude that $n=0$ and $y\in\cO_{F_p}$. If $p\ndivides N^+$, then $R_p={\rm M_2}(\cO_{F_p})$, and hence $[h]=[(\bfone_2,1)]$, as desired. Now assume that $p\mid N^+$. Then $[h_3]$ is represented by $(u,1)$ for some $u\in\SL_2(\cO_{F_p})$. By \eqref{E:Iwasawa}, we may assume $u=\pMX{1}{0}{x}{1}$ or $\pMX{x}{1}{-1}{0}$ for some $x\in\cO_{F_p}$. A direct computation shows that  
\begin{align*}
\varrho(u^{-1})\bfz'=\begin{cases}(\pMX{ \ol{x}\delta_F}{\delta_F}{-\delta_F^{-1}-x\ol{x}\delta_F}{-x\delta_F},\pMX{\ol{x}\delta\delta_F }{\delta\delta_F}{-\ol{\delta}\delta_F^{-1}-x\ol{x}\delta\delta_F}{-x\delta\delta_F})&\text{if }u=\pMX{1}{0}{x}{1},\\
(\pMX{\ol{x}\delta_F^{-1}}{\delta_F^{-1}}{-\delta_F-x\ol{x}\delta_F^{-1}}{-x\delta_F^{-1}},
\pMX{\ol{x}\ol{\delta}\delta_F^{-1}}{\ol{\delta}\delta_F^{-1}}{-\delta\delta_F-x\ol{x}\ol{\delta}\delta_F^{-1}}{-x\ol{\delta}\delta_F^{-1}})&\text{if }u=\pMX{x}{1}{-1}{0}.\end{cases}\end{align*}
It follows that $\varrho(u^{-1})\bfz'\in R_p\oplus R_p$ would imply that either $\delta-\ol{\delta}\in N^+\cO_{F_p}$ or $\delta_F\in N^+\cO_{F_p}$, which is a contradiction. 

It remains to consider the subcases where either $p$ is ramified in $F$ or $p$ is inert in $F$ but ramified in $M$. In this case, $p\ndivides N^+$, $R_p={\rm M}_2(\cO_{F_p})$, and the assumption \eqref{hC} implies $c_p=0$, and write $h^{-1}=(a,\al)$ with $\rmN_{F/\Q}(\det a)=\al^2\in(\Qp^\x)^2$. Let $\frakp$ be the prime of $F$ above $p$. By \eqref{rFK}, $p$ is unramified in $E$, so there exists a uniformizer $\varpi_{E_\frakp}$ of $E_\frakp$ such that
${\rm N}_{E/\Q}(\varpi_{E_\frakp})\in (\Q^\times_p)^2$. It follows that $[h]=[(\bfone_2,1)]$ if we can show that
\beq\label{E:E3}a \in R_p^\x E_\frakp^\x. \eeq
Since $\varrho(h^{-1})\bfz' \xi_{C,p}\in R_p\oplus R_p$, we find that $\al^{-1} a\ol{a}^*\in R_p,\,\al^{-1}a\delta' \ol{a}^*\in R_p$ if and only if \[\al^{-1}a\ol{a}^*\in R_p^\x,\, a \delta' a^{-1}\in R_p.
\]It follows that the conjugation $\gamma(x)=axa^{-1}$ embedds the $\cO_{F_p}$-order $\cO:=\cO_{F_p}[\delta_M']$ into $R_p$ $(\delta_M'=\varsigma_p^{-1}\iota(\delta_M)\varsigma_p)$. By our choice of $\varsigma_p$ (\lmref{varsigma} (ii)), the inclusion $\cO\hookto R_p$ is an optimal embedding, so to prove \eqref{E:E3}, it suffices to show that $\gamma$ is also an optimal embedding of $\cO$ into $R_p$ 
by \lmref{L:SkolemN}. Now suppose that $\gamma:\cO\hookto R_p$ is not optimal. Then one can verify that $p$ must be ramified in $F$ and in $M$, $\delta_F$ is a uniformizer of $F$, and the maximal order $\cO_{E_\frakp}=\cO_{F_p}[\delta_F^{-1}\delta'_M]$. It follows that $\cO$ is the $\cO_{F_p}$-order of conductor $\frakp$, and $\gamma$ is an (optimal) embedding of $\cO_{E_\frakp}$ into $R_p$.  On the other hand, the conjugation $x\mapsto \pDII{1}{\delta_F}x\pDII{1}{\delta_F^{-1}}$ is an embedding of $\cO_{E_\frakp}$ into $R_p$, so by \lmref{L:SkolemN}, we have $a\in R_p^\x  \pDII{1}{\delta_F}E_\frakp^\x$. In particular, $\det a\in \delta_F\rmN_{E/F}(E_\frakp^\x)\cO_{F_p}^\x$, which contradicts to the fact that $\rmN_{F/\Q}(\det a)\in(\Qp^\x)^2$.

{\bf Case (iv) $p\mid N^-$:}  In this case, $p\cO_F=\frakp\ol{\frakp}$ is split in $F$ by our assumption in \subsecref{secYos1} and $\frakp$ is inert in $E$ by \eqref{H'}. Since $R^\times_p = \{ x\in D_p: {\rm n}(x)\in {\mathcal O}^\times_{F_p} \}$ and 
${\rm ord}_\frakp({\rm n}(E^\times_\frakp))=  2\Z$, it is easy to see that every coset in
$ [ \cE_{\bfz,C,p}]$ can be represented by $(\bfone_2,1)$ and $(\pi_{D_p},p)$.

We have proved that cosets of $[\cE_{\bfz,C,p}]$ can be represented by elements in $\wtd \cE_{\bfz,C,p}$, and it is not difficult to show that these cosets represented by $\wtd\cE_{\bfz,C,p}$ are distinct by the same case-by-case analysis as above. We leave the details to the reader. 
\end{proof}
The above proposition suggests the following Heegner condition for $M$:
\beqcd{Heeg}
\text{Each prime factor of $N^-$ (resp. $N^+$) is inert (resp. split) in $M$.}
\eeqcd

Choose $\varsigma_\infty\in D_{0,\infty}^\x$ such that $\Phi_\infty^{-1}(\varsigma_\infty^{-1}\delta_M\varsigma_\infty)=\pDII{\sqrt{-d_M}}{-\sqrt{-d_M}}\in\bbH$. Let ${\mathcal P}_{N}$ be the set of divisors of $N$. For every positive integer $m|C$ and ${\mathcal N}\in {\mathcal P}_{N}$, we define $w_{\mathcal N}, \varsigma, \varsigma^{(m)}\in H_\A$ by 
\begin{align*}
         \varsigma = &\prod_{p\leq\infty} (\varsigma_p,\det\varsigma_p), \quad 
 w_{\mathcal N}   =    \prod_{ p\mid {\mathcal N}, p\mid N^+ } \begin{pmatrix} 0 & 1 \\ -1 & 0 \end{pmatrix}
                              \prod_{ p\mid {\mathcal N}, p\mid N^-} (\pi_{D_p},p),     \\
  \varsigma^{(m)} = & \varsigma 
                              \prod_{ p:\text{split in $M$}}  \begin{pmatrix} 1 & p^{-\Ord_p(m)} \\ 0 & 1 \end{pmatrix} 
                              \prod_{ p: \substack{\text{split in }F \\ \text{non-split in } M }} \begin{pmatrix} p^{-\Ord_p(m)}& 0 \\ 0 & 1 \end{pmatrix}.
                              \end{align*}
Suppose that $M$ satisfies \eqref{Heeg}. By \propref{coset.Y}, $\cE_{\bfz,C}$ is not empty and \begin{align}\label{E:coset.Y}
 \widetilde{\cE}_{\bfz,C} = \{  \varsigma_f^{(m)} w_{\mathcal N} \mid {\mathcal N} \subset {\mathcal P}_N, m\mid C \}.  
\end{align}

\subsection{Bessel periods and toric period integrals}
In this subsection, we express certain normalized Bessel periods in terms of toric period integrals. We begin with some notation. Let $\cO:=R\cap E$ be an $\cO_F$-order of $E$ and for each positive integer $m$, put $\cO_{m}=\cO_F+m\cO$.\footnote{In general, $\cO$ may not be the maximal order $\cO_E$ unless $(\Delta_F,\Delta_K)=1$.} If $L/\Q$ is quadratic, put $\cO_{L,m}=\Z+m\cO_L$. Define the rational number $v_{E/M,m}$ by \[v_{E/M,m}=\sharp(\wh\cO_M^\x/\wh\cO_{M,m}^\x)\cdot \sharp(\wh\cO^\x/\wh\cO_m^\x)^{-1},\]
and define the integer $t_{E,m}$ by 
   \[t_{E,m}=\sharp\stt{xF^\x\in E^\x/F^\x\mid x\in\wh\cO_{m}^\x \wh F^\x}.\]
Note that $t_{E,m}$ divides the order of the torsion subgroup of $\cO_{m}^\x$. 
   
Let $\frakX_K^-$ denote the space of finite order Hecke characters $\phi:K^\x\A^\x\bksl K_\A^\x\to\Zbar^\x$. For each $\phi\in\frakX^-_K$ of conductor $C\cO_K$, we define the normalized Bessel period by
\beq\label{E:Bessel}
\bfB_{\theta_\bff,S,\phi}^{*}:=\frac{C^2\cdot e^{2\pi(1+\delta\ol{\delta})}}{(-2\sqrt{-1})^{k_1+k_2}}\cdot\frac{t_{E,C}}{v_{E/M,C}}\cdot \pair{\bfB_{\theta_\bff,S,\phi}(g_C)}{Q_{S}}_{2k_2}\in\C,\eeq
where $Q_S\in \Z[\frac{1}{\sqrt{d_K}}][X,Y]_{2k_2}$ is defined by \begin{align*}
Q_S:=& ((X,Y)S\begin{pmatrix}X\\Y\end{pmatrix})^{k_2}\cdot (\det S)^{-\frac{k_1+k_2+2}{2}}\\
=&(X^2+(\delta+\ol{\delta})XY+\delta\ol{\delta}Y^2)^{k_2}\cdot (\Im\delta)^{-(k_1+k_2+2)}.\end{align*}
For a Hecke character $\chi:E^\x F^\x_\A\bksl E^\times_{\mathbf A} \to \C^\times$, define the toric period integral by 
\begin{align*}
P({\mathbf f},\chi, h)
 = \int_{ [E^\x/F^\x]} 
        \langle (X_1Y_1)^{k_1}\otimes (X_2Y_2)^{k_2}, {\mathbf f}( t h) \rangle_{2\ul{k}}\cdot
        \chi(t){\rm d}t   .
\end{align*}
By the definition of $\varsigma^{(m)}$ for $m|C$, one can check easily that \[(\varsigma^{(m)})^{-1}(E_\infty^\x\,\x\,\wh\cO_{m}^\x)\varsigma^{(m)}\subset T_2(\R)\,\x\,H(\wh\Z),\]
where $T_2(\R)$ is the group of diagonal matrices in $\bbH$, so we see easily that 
     \[    P({\mathbf f}, \phi\circ {\rm N}_{E/K}, \varsigma^{(m)})= \sharp(\wh\cO^\x/\wh\cO_{m}^\x)^{-1}t_{E,m}^{-1}\cdot \Theta_m({\mathbf f}, \phi\circ{\rm N}_{E/K}),\]
where 
\[ \Theta_m({\mathbf f}, \phi\circ{\rm N}_{E/K}) 
  :=   \sum_{[t]\in E^\times\wh F^\x\backslash \widehat{E}^\times/\wh\cO_{m}^\times} \langle (X_1Y_1)^{k_1} \otimes (X_2Y_2)^{k_2}, {\mathbf f}(t\varsigma^{(m)}) \rangle_{2\ul{k}}\cdot \phi({\rm N}_{E/K}(t)).\]

Let $\frakN^+\mid N^+$ be an ideal of $\cO_F$ and $\frakN=\frakN^+N^-$. For each prime $\frakp\ndivides N^-$ of $\cO_F$, choose an element $\uf_{\frakN,\frakp}\in F_p^\x$ generating $\frakN$. We define the Atkin-Lehner operator $\tau_{\frakN,\frakp}\in \wh D^\x$ as follows: \[\tau_{\frakN,\frakp}=\pMX{0}{1}{-\uf_{\frakN,\frakp}}{0}\text{ if }\frakp\ndivides N^-\text{ and }\tau_{\frakN,\frakp}=\pi_{D_p}\text{ if }\frakp|N^-.\] 
Let $R_{\frakN^+}$ be the Eichler order of level $\frakN^+$. Suppose that $\bff\in\cA_{\ul{k}}(D_\A^\x,\wh R_{\frakN^+}^\x)$ is an eigenform of Atkin-Lehner operators $\tau_{\frakN,\frakp}$ with eigenvalues $\ep_\frakp(\bff)\in\stt{\pm 1}$.  Namely,
$\bff(h\tau_{\frakN,\frakp})=\ep_\frakp(\bff)\cdot \bff(h)$. By definition, $\ep_\frakp(\bff)=1$ if $\frakp\ndivides\frakN$. Put
\beq\label{E:localconstant}\begin{aligned}
   e(\bff,\phi) 
   =&  \prod_{ \substack{ p\mid N^+, \\ p=\frakp: \text{inert in }F }}
                   (1+ \epsilon_{\frakp}(\bff)\phi({\rm N}_{E/K}(\frakP))^{n_\frakp} )     \\
     &  \times \prod_{ \substack{ p\mid N, \\ p=\frakp\ol{\frakp}: \text{split in $F$} }}
                   (1+ \epsilon_{\frakp}(\bff)\epsilon_{\ol{\frakp}}(\bff) \phi({\rm N}_{E/K}(\frakP))^{n_\frakp-n_{\ol{\frakp}}} ),
\end{aligned}\eeq
where $\frakp$ denotes a prime ideal of $\cO_F$ and 
\begin{itemize}
\item $\frakP$ is a prime ideal of $\cO_E$ lying above $\frakp$.
\item $n_\frakp=\Ord_\frakp(\frakN)(=\Ord_\frakp(\frakN^+N^-))$.
\end{itemize}
\begin{prop}\label{Besselper}Suppose that $M$ satisfies \eqref{Heeg} and \eqref{rFK} and let $C$ be an integer satisfying \eqref{hC}. Let $\phi\in\frakX^-_K$ of conductor $C\cO_K$. Then we have \[
 \bfB^*_{\theta_\bff,S,\phi} =  e({\mathbf f}, \phi)\cdot \Theta_C({\mathbf f}, \phi\circ {\rm N}_{E/K}).
\]
\end{prop}\begin{proof}
 Note that 
\[Q_S(X,Y)=\rho_\kappa(\rt\xi_\infty^{-1})((X^2+Y^2)^{k_2})(\det\xi_\infty)^{2k_1+4}\]
for $\xi_\infty=\pMX{\Im\delta}{-\Re\delta}{0}{1}(\Im\delta)^{-1}$. Putting $\varphi^{[0]}_\infty(x) = \langle \varphi_\infty(x), Q_S(X,Y)\rangle_{2k_2}$, by \lmref{lem1Y} and a routine computation, we obtain
\begin{align*}
\varphi^{[0]}(\varrho(\varsigma_\infty)\bfz)
 =& e^{-2\pi(1+\delta\ol{\delta})}\pair{P_{\ul{k}}(\varrho(\varsigma_\infty)\bfz\xi_\infty)}{(X^2+Y^2)^{k_2}}_{2k_2}\\
=&  e^{-2\pi(1+\delta\ol{\delta})}\langle P_{\ul{k}}(\bfone_2,\pDII{\sqrt{-1}}{-\sqrt{-1}}), (X^2+Y^2)^{k_2} \rangle_{2k_2}  \\
=&  e^{-2\pi(1+\delta\ol{\delta})}(-2\sqrt{-1})^{k_1+k_2} (X_1Y_1)^{k_1}\otimes (X_2Y_2)^{k_2}.   
\end{align*}
Therefore, by \propref{P:Bessel1} and \propref{coset.Y}, we find that
\begin{align*}
\bfB_{\theta_\bff,S,\phi}^*=&\frac{e^{2\pi(1+\delta\ol{\delta})}t_{E,C}}{(-2\sqrt{-1})^{k_1+k_2}v_{E/M,C}}\cdot\sum_{h\in\wtd\cE_{\bfz,C}}\mu_h\cdot \int_{[E^\x/F^\x] }
        \langle \varphi_\infty^{[0]}(\bfz ), {\mathbf f}(t h) \rangle_{2\ul{k}}\cdot
        \phi({\rm N}_{E/K}(t)) {\rm d}t \\
        =&\sum_{h\in\wtd\cE_{\bfz,C}}\mu_h\cdot P(\bff,\phi\circ\rmN_{E/K},\varsigma_\infty h),
\end{align*}
where $\mu_h:=\vol(H_\bfz(\A_f)\cap h\cU h^{-1},\rmd a_f)^{-1}$.
Since $\phi$ has conductor $C\cO_K$ and $E/K$ is unramified at prime factors of $C$, one can verify that
\begin{align*}
P({\mathbf f},\phi\circ{\rm N}_{E/K}, \varsigma^{(m)}w_{\mathcal N}) =0
\end{align*}
unless $m=C$. On the other hand, using the proofs in \propref{coset.Y}, one can verify easily that 
\[\mu_{\xi^{(C)}}=\prod_{p|C}\sharp(\frac{H_{\bfz'_p}\cap \cU_p}{H_{\bfz'_p}\cap \xi_p\cU_p\xi_p^{-1}})=\sharp(\wh\cO_M^\x/\wh\cO_{M,C}^\x),\]
where $\bfz'_p=\varrho(\varsigma_p^{-1})\bfz$, $\xi_p=\pMX{1}{C^{-1}}{0}{1}$ if $p$ is split in $M$ and $\xi_p=\pDII{C^{-1}}{1}$ if $p$ is inert in $M$. 
We thus obtain
\[\bfB_{\theta_\bff,S,\phi}^*=\sharp(\wh\cO^\x/\wh\cO_{C}^\x)t_{E,C}\sum_{\cN\subset \cP_N}P(\bff,\phi\circ\rmN_{E/K},\varsigma^{(C)}w_\cN).\]
If $p\mid N^-$, then $p=\frakp\ol{\frakp}$ split in $F/\Q$, and by definition
\begin{align*}
P({\mathbf f},\phi\circ{\rm N}_{E/K}, \varsigma^{(C)} w_p)
= \epsilon_\frakp(\bff) \epsilon_{\ol{\frakp}}(\bff)  
   P({\mathbf f},\phi\circ{\rm N}_{E/K}, \varsigma^{(C)} ).  
\end{align*}
If $p\mid N^+$, then $p$ is split in $M$. Let $\frakP$ be a prime ideal of $\cO_E$ above $p$ and let $\frakp=\cO_F\cap\frakP$. For $x\in F^\x_p$, put
\[d_x=\varsigma_p\pDII{1}{x}\varsigma_p^{-1}.\]
According to the recipe of $\varsigma_p$ in \lmref{varsigma} (i), we see that $d_x\in E_p^\x$ and $\Ord_\frakp(x)=\Ord_{\frakP}(d_x)$. Choosing $\uf_\frakN\in F_p^\x$ with $\Ord_\frakp(\uf_\frakN)=n_\frakp$ for every $\frakp|p$, we have 
\[\bff(h\varsigma^{(C)}w_p)=\bff(d_{\uf_\frakN}^{-1}\varsigma^{(C)}\pMX{0}{1}{-\uf_\frakN}{0})=(\prod_{\frakp|p}\epsilon_\frakp(\bff))\cdot\bff(d_{\uf_\frakN}^{-1}\varsigma^{(C)}).\]
As $\Ord_\frakP(d_{\uf_\frakN})=n_\frakp$, we obtain
\begin{align*}
  P({\mathbf f},\phi\circ{\rm N}_{E/K}, \varsigma^{(C)} w_p)  
=& (\prod_{\frakP|\frakp|p}\epsilon_\frakp(\bff)\phi({\rm N}_{E/K}(\frakP))^{n_\frakp} )\cdot 
   P({\mathbf f},\phi\circ{\rm N}_{E/K}, \varsigma^{(C)} ).  
\end{align*}
This completes the proof. \end{proof}

\begin{rem}\label{Besselsplit}
In the case $F=\Q\oplus\Q$, $\frakN^+=(N_1^+,N_2^+)$, we have ${\mathbf f} = {\mathbf f}_1\otimes {\mathbf f}_2$, where $\bff_i$ is a $\cW_{k_i}(\C)$-valued modular form on $(D_0\ot\A)^\x$ for $i=1,2$. For a finite order Hecke character $\phi: K^\x\A^\x\bksl K^\times_{\mathbf A}\to \C^\x$, we put 
\begin{align*}
    P({\mathbf f}_i, \phi, h)
 = \int_{K^\times  \Q^\times_{\mathbf A}\backslash K^\times_{\mathbf A} }   
        \langle (X_iY_i)^{k_i}, {\mathbf f}_i(  t h ) \rangle_{2k_i}
        \phi(t) {\rm d}t \quad (i=1,2).
\end{align*}
Then one verifies that
\begin{align*}
     P({\mathbf f}, \phi\circ {\rm N}_{E/K}, \varsigma^{(C)})   
  = & P({\mathbf f}_1, \phi, \varsigma^{(C)}) P({\mathbf f}_2, \phi^{-1}, \varsigma^{(C)}).
\end{align*}
\end{rem}

\section{The non-vanishing of Bessel periods}\label{secFC}
\subsection{Integrality of Yoshida lifts}\label{SS:normalizationY}
Let $\ell\ndivides 2N$ be a rational prime and fix an isomorphism $\iota_\ell:\C\iso \Qbar_\ell$. Let $\Qbar$ be the algebraic closure of $\Q$ in $\C$ and let $\lambda$ be the place of $\bar{\Q}$ induced by $\iota_\ell$. Let $\cO_\lam$ be the completion of the algebraic integers $\ol{\Z}$ along $\lam$ and let $\frakl$ be the prime ideal of $\cO_F$ lying under $\lam$. Embedd $F\ot\Q_\ell\hookto F_\frakl\oplus F_\frakl,\,x\mapsto (\iota_\ell(x),\iota_\ell(\ol{x}))$. Recall that in \subsecref{SS:DATA} we have chosen a real quadratic field $F'$ in which $\ell$ splits and fixed an isomorphism $\Phi_{F'}:\bbH_\Q\ot F'\iso D_0\ot F'$ such that $\Phi_{F'}(\cO_{\bbH_\Q}\ot\Z_\ell)=\cO_{D_0}\ot\Z_\ell$. Then $\Phi_{F'}^{-1}$ gives rise to a morphism $\tau_{\ul{k},\frakl}:=\tau_{\ul{k}}\circ\Phi_{F'}^{-1}:D_\ell^\x\to\Aut\cW_{\ul{k}}(\Qbar_\ell)$ induced by 
\[D_\ell\iso \bbH\ot F\ot\Q_\ell\hookto {\rm M}_2(F_\frakl(\sqrt{-1})\oplus F_\frakl(\sqrt{-1}))\hookto {\rm M}_2(\Qbar_\ell\oplus \Qbar_\ell).\] 
By construction, we have
$\tau_{\ul{k},\frakl}:R_\ell^\x\iso(\cO_{\bbH_\Q}\ot\Z_\ell\ot\cO_F)^\x
\to\Aut\cW_{\ul{k}}(\cO_\lam)$ and 
\[\tau_{\ul{k},\frakl}(\gamma)=\tau_{\ul{k}}(\gamma)\in\Aut\cW_{\ul{k}}(\Qbar)\text{ for }\gamma\in D^\x.\]
Define the $\ell$-adic avatar $\widehat{\mathbf f}: \wh D^\x\to {\mathcal W}_{\ul{k}}(\Qbar_\ell)$  of ${\mathbf f}$ by  
$\widehat{\mathbf f}(h) = \tau_{\ul{k},\frakl}(h^{-1}_\ell){\mathbf f}(h).$ By definition, we can verify that
\[\wh \bff(\gamma h uz)=\tau_{\ul{k},\frakl}(u_\ell^{-1})\wh \bff(h)\quad(\gamma\in D^\x,\,u\in \wh R^\x,\,z\in \wh F^\x). \] Hence the values of $\wh\bff$ are determined by those at representatives of the finite double coset $D^\x\bksl \wh D^\x/\wh R^\x$, and we can normalize ${\mathbf f}$ by multiplying a scalar in $\Qbar_\ell^\x$ so that $\widehat{\mathbf f}$ takes values in ${\mathcal W}_{\ul{k}}(\cO_\lam)$ and $\widehat{\mathbf f}\not\equiv 0\pmod{\lambda}$. In what follows, we assume $\bff$ is normaliazed as above.

\begin{prop}Suppose that $\ell>2k_1$. The classical Yoshida lift $\theta_{\bff}^*$ has $\lam$-integral Fourier expansion. \end{prop}
\begin{proof}
From the Fourier expansion $\sum_S \bfa(S)q^S$ of $\theta_\bff^*$ in \propref{P:FCYoshida}, we have \begin{align*}\bfa(S)=&\sum_{[h_f]\in [\cE_{\bfz}]}w_{\bfz,h_f}\pair{P_{\ul{k}}(\bfz)}{\bff(h_f)}_{2\ul{k}}\\
=&\sum_{[h_f]\in [\cE_{\bfz}]}w_{\bfz,h_f}\pair{P_{\ul{k}}(\varrho(h_\ell^{-1}) z)}{\wh\bff(h_f)}_{2\ul{k}}.\end{align*}
We note that $P_{\ul{k}}(\varrho(h_\ell^{-1}) z)\in\cO_\lam[X_1,Y_1]_{2k_1}\ot\cO_\lam[X_2,Y_2]_{2k_2}\ot\cO_\lam[X,Y]_{2k_2}$ since $h_f\in\cE_{\bfz}$ implies that $\varrho(h_\ell^{-1} z)\in R_\ell=\Phi_{F'}(\cO_{\bbH_\Q}\ot\Z_\ell)\ot\cO_{F}$, and $P_{\ul{k}}(x)$ is a polynomial on $\bbH^{\oplus 2}$ with coefficients in $\Z$ which takes value in $\Z[1/2]$ on $\cO_{\bbH_\Q}$. Combined with the fact that the pairing $\pair{\cdot}{\cdot}$ on $\cW_{\ul{k}}(\cO_\lam)$ takes value in $\cO_\lam$ if $\ell>2k_1$, we see immediately that $\bfa(S)\in\cL_\kappa(\cO_\lam)$. \end{proof}

Now we fix a prime $\ell>2k_1$ and retain the notation $M,K,\delta,\bfz,C,\ldots$ and the hypotheses \eqref{rFK}, \eqref{Heeg} and \eqref{hC} in the previous section. We relate the non-vanishing of Bessel periods (modulo $\lam$) to the non-vanishing of Fourier coefficients of Yoshida lifts.
\begin{lm}\label{L:BesselFC}Assume that $\ell\ndivides 2C\Delta_K$. Let $\phi\in\frakX_K^-$ of conductor $C\cO_K$. Then $\bfB^*_{\theta_\bff,S_\bfz,\phi}\in\cO_\lam$, and if $\bfB^*_{\theta_\bff,S_\bfz,\phi}\not\con 0\pmod{\lam}$, then there exists some $S'\in\Lam_2$ such that $\det S'=C^2\Delta_K/4$ and
\[\bfa(S')\not \con 0\pmod{\lam}.\]
\end{lm}
\begin{proof}Let $S=S_\bfz$. By definitions \eqref{E:dfnBessel} and \eqref{E:Bessel}, the normalized Bessel period $\bfB^*_{\theta_\bff,S,\phi}$ is equal to 
\beq\label{E:4}
  \frac{C^2\cdot e^{2\pi(1+\delta\ol{\delta})}t_{E,C}}{(-2\sqrt{-1})^{k_1+k_2}v_{E/M,C}}\sum_{[t]}\frac{v_{E/M,C}}{t_{K,C}}\cdot \pair{\bfW_{\theta_\bff,S}(j(t)g_C)}{Q_{S}}_{2k_2}\cdot \phi(\rmN_{E/K}(t)),
\eeq
where $[t]$ runs over the finite double cosets $E^\x \wh E^\x_{0}\bksl \wh E^\x/\wh\cO^\x_{C}$ 
and  
\[t_{K,C}:=\sharp\stt{xE^\x_0\in E^\x/E_0^\x\mid x\in \wh\cO_{C}^\x \wh E_0^\x}.\]
Note that $t_{K,C}$ divides the order of the torsion subgroup of $\cO_{K,C}^\x$ and hence $t_{K,C}\in\Z_\ell^\x$ as  $\ell\ndivides\Delta_K $. By strong approximation, for $t\in\wh E^\x$ we can decompose $\Psi(\rmN_{E/K}(t))=\al_t u_t$ with $\al_t\in\GL_2(\Q)$ and $u_t\in \GL_2(\wh\Z)$. Put $\gamma_t:=\sqrt{\det\al_t}^{-1}\al_t\xi_C$ and $S^{\gamma_t}=\rt\gamma_t S\gamma_t$. Applying  \eqref{E:FC3} and the computation in \propref{P:FCYoshida}, we can verify that
\[\bfW_{\theta_\bff,S}(j(t)g_C)=\rho_\kappa(\rt\gamma_t^{-1})\bfa(S^{\gamma_t})e^{-2\pi(1+\delta\ol{\delta})};\]hence \eqref{E:4} is equal to 
\[ \frac{C^2t_{E,C}}{(-2\sqrt{-1})^{k_1+k_2}t_{K,C}}\sum_{[t]}\pair{\bfa(S^{\gamma_t})}{(\det\gamma_t^{2k_1+4})\cdot\rho_\kappa(\rt\gamma_t)Q_{S}}_{2k_2}\cdot \phi(\rmN_{E/K}(t)).\]
A little computation shows that \[(\det\gamma_t^{2k_1+4})\cdot\rho_\kappa(\rt\gamma_t)Q_{S}=
Q_{S^{\gamma_t}},\]
so we obtain
\beq\label{E:BesselFC}\bfB^*_{\theta_\bff,S,\phi}= \frac{C^2}{(-2\sqrt{-1})^{k_1+k_2}}\cdot\frac{t_{E,C}}{t_{K,C}}\cdot\sum_{[t]}\pair{\bfa(S^{\gamma_t})}{Q_{S^{\gamma_t}}}_{2k_2}\cdot\phi(\rmN_{E/K}(t)).\eeq
Note that $S^{\gamma_t}\in\cH_2(\Q)$ and $\det S^{\gamma_t}=C^2\Delta_K/4$. If $\bfa(S^{\gamma_t})\not =0$, then $S^{\gamma_t}\in\Lam_2$, and $Q_{S^{\gamma_t}}\in\Z[\frac{1}{C\Delta_K}][X,Y]$. This shows that
\[\pair{\bfa(S^{\gamma_t})}{Q_{S^{\gamma_t}}}_{2k_2}\in \cO_\lam[\frac{1}{C\Delta_K}].\]
By \eqref{E:BesselFC}, we conclude that if $\bfB^*_{\theta_\bff,S,\phi}\not \con 0\pmod{\lam}$, then $\bfa(S^{\gamma_t})\not \con 0\pmod{\lam}$
for some $t\in\wh E^\x$. This completes the proof.\end{proof}

\subsection{The non-vanishing of Yoshida lifts}
We investigate the problem of the non-vanishing of Yoshida lifts modulo $\lam$ in the case of $F=\Q\oplus\Q$. Let $(f_1,f_2)$ be a pair of elliptic newforms of weight $(k_1+2,k_2+2)$ and level $(\Gamma_0(N_1^+N^-),\Gamma_0(N_2^+N^-))$. Let $N={\rm l.c.m}(N_1^+N^-,N_2^+N^-)$ and $\frakN^+=(N_1^+,N_2^+)$. Suppose further that $\bff\in\cA_{\ul{k}}(D_\A^\x,\wh R^\x_{\frakN^+})$ is the normalized newform associated with $(f_1,f_2)$ in the sense of \subsecref{subsec:fromsonH}.
\begin{thm}\label{T:nonvanYoshida}
Suppose that 
\begin{itemize}
\item[(LR)] For every $q\mid N$ with $q=\frakq\ol{\frakq}$ split in $F$ and ${\rm ord}_\frakq({\mathfrak N}) = {\rm ord}_{\ol{\frakq}}({\mathfrak N}) >0$, \[\epsilon_\frakq(\bff)=\epsilon_{\ol{\frakq}}(\bff).\]
\end{itemize}
Assume that $\ell$ satisfies the following conditions\begin{enumerate}
\item \label{nonvanBessel(ia)}
          $\ell > 2k_1$ and $\ell \nmid 2N $; 
\item \label{nonvanBessel(ib)}
          the residual $\lam$-adic Galois representation $\bar{\rho}_{f_i, \lambda}: {\rm Gal}(\bar{\Q}/\Q) \to {\rm GL}_2(\bar{\mathbb F}_\ell)\ (i=1,2)$ is absolutely irreducible. 
\end{enumerate}
Then $\theta_{\bff}^*=\sum_{S\in\Lam_2}\bfa(S)q^S\not \con 0\pmod{\lam}$.  Moreover, for every imaginary quadratic field $K$ with \eqref{Heeg} and $(\ell,\Delta_K)=1$, there exist infinitely many $S\in\Lam_2$ such that $\Q(\sqrt{-\det S})=K$ and 
$\bfa(S)\not \con 0\pmod{\lam}$. \end{thm}
\begin{proof}Let $K$ be as above.  We choose a prime $p\ndivides\ell N\Delta_K$ and let $\phi\in\frakX_K^-$ of conductor $p^n\cO_K$. Then $M=K$ satisfies \eqref{Heeg}, \eqref{rFK}, and $C=p^n$ satisfies \eqref{hC}. By \propref{Besselper} ($E=K\oplus K$) and  \remref{Besselsplit}, we see that $\bff=\bff_1\ot\bff_2$ and 
\begin{align*}
 \bfB^*_{\theta_\bff,S,\phi}=e({\mathbf f}, \phi) \cdot\Theta_{p^n}({\mathbf f}_1, \phi) \Theta_{p^n}({\mathbf f}_2, \phi^{-1}),   
\end{align*}
where 
\begin{align*}
       \Theta_{p^n}({\mathbf f}_i, \phi^\pm) 
  := &  \sum_{[t]\in K^\x\bksl \wh K^\x/\wh\cO_{K,p^n}^\x} \langle (X_iY_i)^{k_i} , {\mathbf f}_i(t\varsigma^{(p^n)}) \rangle_{2k_i} \cdot \phi^\pm(t).
\end{align*} 

Now under our assumptions, one can show that $\Theta_{p^n}({\mathbf f}_1,\phi)$ and $\Theta_{p^n}({\mathbf f}_2,\phi^{-1})$ are both nonzero modulo $\lambda$ for all but finitely many $\phi\in {\mathfrak X}_K^-$ of $p$-power conductor by the same arguments in \cite[Theorem 5.9]{ch15} (replace $F_\ell(g)$ by $F_\ell^0(g):=\pair{(X_iY_i)^{k_i}}{\bff_i(g)}$ in the proof). In addition, the condition (LR) implies that $e({\mathbf f}, \phi) \not\equiv 0 \ {\rm mod}\ \lambda$ as long as $\phi$ is sufficiently ramified.
Therefore, the theorem follows from \lmref{L:BesselFC} immediately.\end{proof}

\begin{thank}This work was done while the second author was a postdoctoral fellow in Taida Institute
for Mathematical Sciences and National Center for Theoretical Sciences. He would like
to thank for their supports and hospitalities. The authors are grateful to the referee for suggestions on the improvement of the paper.
\end{thank}

\bibliographystyle{amsalpha}
\bibliography{yoshidabib}

\end{document}